\documentclass[11pt, reqno]{amsart}

\usepackage{amsmath, amsfonts, amssymb, amsbsy, bigstrut, graphicx, enumerate,  upref, longtable, comment, comment}

\usepackage{pstricks}

\usepackage[breaklinks]{hyperref}

\usepackage[numbers, sort&compress]{natbib}

\newcommand{\mg}{\mathcal{G}}
\newcommand{\ep}{\epsilon}

\newcommand{\mc}{\mathcal{C}}
\newcommand{\ms}{\mathcal{S}}

\newtheorem{thm}{Theorem}[section]
\newtheorem{lmm}[thm]{Lemma}
\newtheorem{cor}[thm]{Corollary}

\newtheorem{defn}[thm]{Definition}

\theoremstyle{definition}

\newtheorem{ex}[thm]{Example}

\newcommand{\avg}[1]{\bigl\langle #1 \bigr\rangle}

\newcommand{\md}{\mathcal{D}}
\newcommand{\mf}{\mathcal{F}}

\newcommand{\cp}{\mathcal{P}}

\newcommand{\mv}{\mathcal{V}}

\newcommand{\rr}{\mathbb{R}}
\newcommand{\smallavg}[1]{\langle #1 \rangle}

\newcommand{\mn}{\mathcal{N}}

\newcommand{\tT}{\tilde{T}}

\newcommand{\mfn}{\mathfrak{N}}
\newcommand{\mfc}{\mathfrak{C}}

\newcommand{\ts}{\tilde{s}}
\newcommand{\tp}{\tilde{\P}}

\numberwithin{equation}{section}
\newcommand{\tx}{\tilde{X}}
\newcommand{\ty}{\tilde{Y}}

\usepackage[utf8]{inputenc}
\usepackage{amsmath}
\usepackage{amsfonts}
\usepackage{bbm}
\usepackage{mathtools}
\usepackage{tikz}
\usepackage{amsthm}
\usepackage[english]{babel}
\usepackage{fancyhdr}
\usepackage{amssymb}
\usepackage{enumitem}
\usepackage{qtree}
\usepackage{mathrsfs}

\DeclareMathOperator{\tr}{tr}

\newcommand{\Z}{\mathbb{Z}}

\newcommand{\R}{\mathbb{R}}

\newcommand{\1}{\mathbbm{1}}

\newcommand{\E}{\mathbb{E}}

\renewcommand{\P}{\mathbb{P}}

\renewcommand{\d}{\mathrm{d}}

\renewcommand{\tilde}{\widetilde}

\begin{document}

\title{Average Gromov hyperbolicity and the Parisi ansatz}
\author{Sourav Chatterjee}
\address{Department of Statistics, Stanford University, 390 Jane Stanford Way, Stanford, CA 94305}
\email{souravc@stanford.edu}
\author{Leila Sloman}
\address{Department of Mathematics, Stanford University, 450 Jane Stanford Way, Building 380, Stanford, CA 94305}
\email{lsloman@stanford.edu}
\thanks{Sourav Chatterjee's research was partially supported by NSF grant DMS-1855484}
\thanks{Leila Sloman's research was partially supported by NSF grant  DGE-1656518.}
\keywords{Hyperbolic metric space, Gromov hyperbolicity, ultrametricity, spin glass, negative curvature}
\subjclass[2010]{51M10, 53C23, 60K35, 82B44}

\begin{abstract}
Gromov hyperbolicity of a metric space measures the distance of the space from a perfect tree-like structure. The measure has a ``worst-case'' aspect to it, in the sense that it detects a region in the space which sees the maximum deviation from tree-like structure. In this article we introduce  an ``average-case'' version of Gromov hyperbolicity, which detects whether the ``most of the space'', with respect to a given probability measure, looks like a tree. The main result of the paper is that if this average hyperbolicity is small, then the space can be approximately embedded in a tree. The proof uses  a weighted version of Szemer\'edi's regularity lemma from graph theory. The result applies to Gromov hyperbolic spaces as well, since average hyperbolicity is bounded above by Gromov hyperbolicity.  As an application, we give a construction of hierarchically organized pure states in any model of a spin glass that satisfies the Parisi ultrametricity ansatz.   
\end{abstract}


\maketitle


\section{Gromov hyperbolicity}
Let $(S,d)$ be a metric space. The {\it Gromov product} of two points $x,y\in S$ with respect to a third point $z\in S$ is defined as
\[
(x,y)_z := \frac{1}{2}(d(x,z)+d(y,z)-d(x,y)). 
\]
Note that by the triangle inequality, the Gromov product is always nonnegative. The space is called $\delta$-hyperbolic (as defined by Gromov~\cite{gromov87}) if for any four points $x,y,z,w\in S$, 
\begin{align}\label{fourpoint}
(x,y)_w \ge \min\{(x,z)_w, (y,z)_w\} -\delta. 
\end{align}
The smallest $\delta$ for which this is satisfied is known as the Gromov hyperbolicity of $(S,d)$. The condition \eqref{fourpoint} is known as Gromov's four point condition. It is not hard to show that if \eqref{fourpoint} is satisfied for all $x,y,z$ for a given $w_0$, then it can be shown that it is satisfied for all $w$ with $2\delta$ in place of $\delta$. Thus, we may equivalently define hyperbolicity using a three point condition, by fixing $w$. If \eqref{fourpoint} is satisfied for all $x,y,z$ for some fixed $w$, then we say that the space is $\delta$-hyperbolic with base point $w$. 

The notion of hyperbolic metric spaces is closely related to the notion of real trees. If $(T,\rho)$ is a metric space and $x,y\in T$, an {\it arc} from $x$ to $y$ is the image of a topological embedding $\gamma: [a,b]\to T$ with $\gamma(a) = x$ and $\gamma(b)=y$, where $[a,b]$ is a closed interval in $\R$ (allowing the possibility that $a=b$). A {\it geodesic segment} from $x$ to $y$ is the image of an isometric embedding $\gamma:[a,b]\to T$ with $\gamma(a)=x$ and $\gamma(b)=y$. A metric space $(T,\rho)$ is called a {\it real tree} if for any $x,y\in T$, there exist a unique arc from $x$ to $y$, and this arc is a geodesic segment. A real tree with a distinguished point $r\in T$ is called a rooted real tree with root $r$. 

The most elementary connection between Gromov hyperbolicity and real trees is that a metric space is $0$-hyperbolic if and only if it is isometric to a subset of a real tree.  Now suppose that a metric space $(S,d)$ is $\delta$-hyperbolic for some small but nonzero $\delta$. Is it approximately isometric to a subset of a real tree, in some sense? The following result shows that this is true when $S$ has finite cardinality, with an error proportional to $\delta\log |S|$.
\begin{thm}[\citet{ghysdelaharpe90}]\label{ghysthm}
Let $(S,d)$ be a $\delta$-hyperbolic metric space with base point $w$ and finite cardinality. Let $k$ be a positive integer such that $|S|\le 2^k + 2$. Then there exists a real tree  $(T, \rho)$ with root $r$ and a map $\Phi:S\to T$ such that for all $x\in S$, $d(x,w) = \rho(\Phi(x), r)$, and for all $x,y\in S$, $d(x,y)-2k\delta\le \rho(\Phi(x),\Phi(y))\le d(x,y)$.
\end{thm}
It is known that the error of order $\delta \log |S|$ in the above theorem cannot be improved~\cite{bowditch06}. In particular, it is not possible to have a quasi-isometry where the discrepancy depends solely on $\delta$. 


The notion of Gromov hyperbolicity,  introduced by Gromov in a group-theoretic context, has found great success in many areas of mathematics and even in science and engineering. There are many examples of metric spaces, both in theory and practice, that are almost tree-like but not exactly so. Gromov hyperbolicity is a great way to understand and study such examples. 

Still, there is one aspect of Gromov hyperbolicity that is sometimes problematic when one ventures outside the domain of very regular objects coming from pure mathematics. It is the fact that the four point condition \eqref{fourpoint} is a worst-case condition: The space is not $\delta$-hyperbolic if there is even a single four-tuple $(x,y,z,w)$ for which~\eqref{fourpoint} fails. There are examples from statistical physics and probability theory where \eqref{fourpoint} holds for most, but not all four-tuples~\cite{panchenko13}. Here ``most'' is in terms of a probability measure on the space. Similar examples arise in the applied sciences, such as in the analysis of social networks~\cite{adm14} and phylogeny reconstruction~\cite{chakerianholmes12}. 

For these reasons, one may naturally wonder whether the condition \eqref{fourpoint} may be replaced by some kind of an averaged version. This has, indeed, been proposed recently in some physics papers (such as \cite{adm14}), but these proposals have not been mathematically analyzed. The goal of this manuscript is to fill this gap: We define a natural notion of average Gromov hyperbolicity, and prove an analog of Theorem \ref{ghysthm} for this measure. Interestingly, unlike Theorem \ref{ghysthm}, this result has no dependence on the size of $S$. The proof is  more involved than the proof of Theorem \ref{ghysthm}, using a weighted version of Szemer\'edi's regularity lemma from graph theory. We apply this theorem to show that hierarchically organized pure states can be constructed in any model of a spin glass that satisfies the Parisi ultrametricity ansatz. 

\section{Main result}\label{backsec}
We will go beyond metric spaces in our definition of average hyperbolicity. Let $S$ be a set equipped with a countably generated $\sigma$-algebra $\mf$ and a probability measure $\P$ defined on $\mf$. Let $b$ be a positive real number and $s:S\times S \to [0,b]$ be a measurable function satisfying  $s(x,y)=s(y,x)$ for all $x,y\in S$. We will say that $s$ is a ``similarity function''. Intuitively, $s(x,y)$ measures the similarity between two points $x$ and $y$. Similarity functions generalize the notion of Gromov product: If $S$ has finite diameter with respect to a separable metric and is endowed with the Borel $\sigma$-algebra generated by this metric,  the Gromov product $(x,y)_w$ is a similarity function for any base point $w\in S$. 
\begin{defn}
We will say that $(S, \mf, \P, s)$ is $\delta$-hyperbolic if 
\[
\textup{Hyp}(S,\mf, \P, s) := \E(\min\{s(X,Z), s(Y,Z)\}-s(X,Y))_+\le \delta,
\]
where $x_+$ denotes the positive part of a real number $x$, and $X,Y,Z$ are i.i.d.~$S$-valued random variables with law $\P$. 
\end{defn}
It is not hard to show that $(S, \mf, \P, s)$ is $0$-hyperbolic in the above sense if and only if there is a real tree $(T,\rho)$ with root $r$ and set of leaves $S$, such that for all $x,y$ in the support of $\P$, we have $s(x,y) = (x,y)_r$, where $(x,y)_r$ is the Gromov product of $x$ and $y$  under the metric $\rho$, with respect to the base point $r$.  We will now generalize this result  when $(S,\mf, \P,s)$ is $\delta$-hyperbolic for some small $\delta$. First, recall that a graph-theoretic tree, henceforth simply called a tree, is a  connected undirected graph without self-loops or closed paths.  A rooted tree is a tree where one distinguished node is called the root. A node of a rooted tree is called a leaf if it is not the root and it has degree one.
\begin{defn}\label{treedef}
We will say that a tree $T$ with root $r$ is compatible with $(S,\mf)$ if the following three conditions are satisfied:
\begin{enumerate}
\item[\textup{(i)}] $S$ is the set of leaves of $T$, 
\item[\textup{(ii)}] $T\setminus S$ is a finite set, and 
\item[\textup{(iii)}] for any node $v\in T\setminus S$, the set of leaves that are the descendants of $v$ is a measurable subset of $S$. 
\end{enumerate}
\end{defn}
Clearly, any tree that is compatible with $(S,\mf)$ gives a hierarchical clustering of $S$, such that the number of clusters is finite and each cluster is measurable. Conversely, any such clustering defines a compatible tree. An example is shown in Figure \ref{treepic}.

If $T$ is a compatible tree with root $r$, and $x,y\in S$, we denote by $(x,y)_r$ the Gromov product of $x$ and $y$ under the graph distance on $T$, with respect to the base point $r$. From the definition of the Gromov product, it is easy to see that $(x,y)_r$ is the number of edges in the intersection of the paths leading from $x$ and $y$ to $r$ (see Figure \ref{treepic}).
\begin{figure}
\begin{center}
\begin{tikzpicture}[scale = 2]
\draw [fill] (0,0) circle [radius = 0] node [black,above] {$r$};
\draw (0,0) to (-1.5,-.5);
\draw (-1.5, -.5) to (-2,-1);
\draw [fill] (-2,-1) circle [radius = 0.04];
\draw (-1.5, -.5) to (-1,-1);
\draw [fill] (-1,-1) circle [radius = 0.04];
\draw [ultra thick] (0,0) to (0, -.5);
\draw (0,0) to (1.5, -.5);
\draw (1.5, -.5) to (2,-1);
\draw [fill] (2,-1) circle [radius = 0.04];
\draw (1.5, -.5) to (1,-1);
\draw [fill] (1,-1) circle [radius = 0.04];
\draw [ultra thick] (0, -.5) to (-.5,-1);
\draw (0, -.5) to (.5,-1);
\draw (-.5, -1) to (-.75, -1.5);
\draw (-.5, -1) to (-.25, -1.5);
\draw [fill] (-.25,-1.5) circle [radius = 0.04] node [black,below] {$y$};
\draw (.5, -1) to (.75, -1.5);
\draw [fill] (.75,-1.5) circle [radius = 0.04];
\draw (.5, -1) to (.25, -1.5);
\draw [fill] (.25,-1.5) circle [radius = 0.04];
\draw (-.75, -1.5) to (-1,-2);
\draw [fill] (-1,-2) circle [radius = 0.04] node [black,below] {$x$};
\draw (-.75, -1.5) to (-.5,-2);
\draw [fill] (-.5,-2) circle [radius = 0.04];
\end{tikzpicture}
\caption{A tree $T$ compatible with $S$, with root $r$. The leaves of $T$, shown using dots, are the elements of $S$. The number of edges in the thickened path equals the Gromov product $(x,y)_r$.\label{treepic}}
\end{center}
\end{figure}
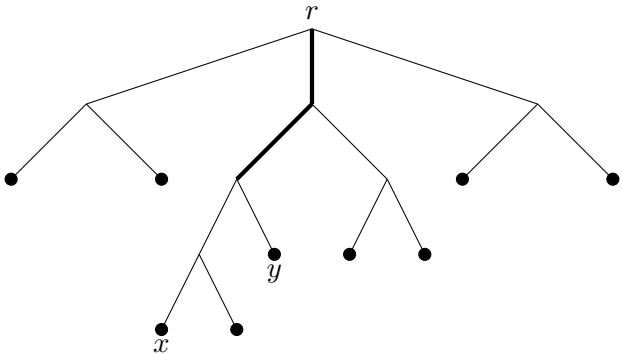
\begin{defn}
We will say that $(S,\mf, \P, s)$ is $\delta$-tree-like if 
\begin{align*}
\textup{Tree}(S,\mf, \P, s) &:= \inf_{T, \alpha}\E|s(X,Y)- \alpha(X,Y)_r|\le \delta,
\end{align*}
where $X$ and $Y$ are independent $S$-valued random variables with law $\P$, and the infimum is taken over over all $\alpha \ge 0$ and all rooted trees $T$ that are compatible with $(S,\mf)$. Here $r$ is the root of $T$ and $(X,Y)_r$ is the Gromov product of $X$ and $Y$ under the graph distance on $T$, with respect to the base point $r$.  
\end{defn}
Note that in the above definition, it follows easily by the definition of compatibility that $(X,Y)_r$ is a bounded and measurable random variable, and therefore the expectation is well-defined. 

The following theorem is the main result of this paper. It shows that $\text{Hyp}(S,\mf, \P, s)$ is small if and only if $\text{Tree}(S,\mf, \P, s)$ is small. 
\begin{thm}\label{mainresult0}
Let $S$, $\mf$, $\P$, $s$ and $b$ be as above. Then given any $\ep>0$, there is some $\delta>0$  depending only on $\ep$ and $b$, such that if $\textup{Hyp}(S,\mf,\P, s)< \delta$, then $\textup{Tree}(S,\mf,\P, s)< \ep$. Conversely, given any $\ep>0$ there is some $\delta>0$  depending only on $\ep$ and $b$, such that if $\textup{Tree}(S,\mf,\P, s)< \delta$, then $\textup{Hyp}(S,\mf,\P, s)< \ep$.
\end{thm}
The above theorem is a generalization of Theorem \ref{ghysthm} to the setting of average hyperbolicity. The statement is more satisfactory than that of Theorem \ref{ghysthm} in that the error has no dependence on the size of $S$. In particular, it remains meaningful even if $S$ has infinite cardinality. Moreover, since Gromov hyperbolicity is obviously greater than or equal to the average hyperbolicity with respect to any probability measure (where the similarity function is the Gromov product with respect to a base point), Theorem \ref{mainresult0} immediately implies the following corollary  about Gromov hyperbolic metric spaces. 
\begin{cor}\label{maincor}
Let $(S,d)$ be a separable metric space with finite diameter $D$, which is $\delta$-hyperbolic with respect to a base point $w$ in Gromov's sense. Then for any probability measure $\P$ defined on the Borel $\sigma$-algebra of $S$, there is a rooted tree $T$ with root $r$ that is compatible with $S$ in the sense of Definition \ref{treedef}, and a number $\alpha \ge 0$, such that 
\[
\iint |(x,y)_w- \alpha(x,y)_r| \d \P(x)\d \P(y)\le \ep(\delta, D),
\]
where $\ep(\delta, D)$ is a number depending only on $\delta$ and $D$ which tends to $0$ as $\delta \to0$. Here $(x,y)_w$ is the Gromov product of $x$ and $y$ under the metric $d$, with respect to the base point $w$, and $(x,y)_r$ is the Gromov product of $x$ and $y$ under the graph distance on $T$, with respect to the base point~$r$. 
\end{cor}
The dependence of $\delta$ on $\ep$ in Theorem \ref{mainresult0} is an important question. The proof given in this paper uses Szemer\'edi's regularity lemma~\cite{szemeredi78}, and therefore cannot be expected to yield useful bounds. It would be very interesting to figure out whether Szemer\'edi's lemma can be bypassed in the proof of Theorem \ref{mainresult0}. If that is possible, then one can at least hope to get reasonable bounds on $\delta$ in terms of $\ep$. 

To see why something like the regularity lemma may be needed, recall the triangle removal lemma of \citet{rs78}: If a simple graph on $n$ vertices has $o(n^3)$ triangles, then it is possible to delete $o(n^2)$ edges and make it triangle-free. The original proof of this result used Szemer\'edi's regularity lemma, and although we now have other approaches~\cite{fox11}, there is still no simple proof of this seemingly simple-sounding claim. Theorem \ref{mainresult0} is a result of a similar spirit, since it asserts that a space which is nearly tree-like in most places may be slightly modified to yield a space that is exactly embeddable in a tree.  

\section{Hyperbolicity and the Parisi ansatz}
In this section we study a well-known class of systems that arise in statistical physics and probability theory that are hyperbolic in the average sense but not in Gromov's sense. 

A {\it spin glass} model assigns a {\it random} probability measure $\mu_n$ on a set $\Sigma_n$, where $\Sigma_n$ is usually the hypercube $\{-1,1\}^n$ or the sphere of radius $\sqrt{n}$ centered at the origin in $\rr^n$. Throughout the rest of this section, we will assume that $\Sigma_n$ is either of these two. The specific definitions of these measures are not particularly relevant for this discussion, so we will not bother to introduce them here. The interested reader may consult~\cite{mpv87, talagrand11a, talagrand11, panchenko13b}. The measure $\mu_n$ is called the Gibbs measure, and the set $\Sigma_n$ is called the configuration space. 

An important quantity in spin glass theory is the {\it overlap} between two configurations $\sigma^1,\sigma^2\in \Sigma_n$, defined as 
\[
R_{1,2} :=\frac{1}{n}\sum_{i=1}^n \sigma_i^1\sigma_i^2\in [-1,1].
\]
The usual convention in the literature is to denote by $R_{i,j}$ the overlap between $\sigma^i$ and $\sigma^j$, where $\sigma^1,\sigma^2,\ldots$ is an i.i.d.~sequence of configurations drawn from the Gibbs measure $\mu_n$. 
It was famously conjectured by \citet{parisi79, parisi80} that certain spin glass models have the property that in the ``$n=\infty$ limit'', $R_{1,2}$ is greater than or equal to the minimum of $R_{1,3}$ and $R_{2,3}$ with probability one. This is known as the Parisi ultrametricity ansatz. Following a long line of deep contributions by various authors~\cite{ac98, gg98, aa09, talagrand06}, the Parisi conjecture was finally proved by Panchenko \cite{panchenko13} for spin glass models that satisfy a certain set of equations known as the generalized Ghirlanda--Guerra identities~\cite{gg98, panchenko10, talagrand03}. The precise statement of  Panchenko's theorem  is that in such models, for any $\ep>0$, 
\begin{align}\label{panchenko}
\lim_{n\to\infty} \E\smallavg{\1_{\{R_{1,2}\ge \min \{R_{1,3}, R_{2,3}\} - \ep\}}} = 1,
\end{align}
where $\smallavg{\cdot}$ denotes expectation with respect to the Gibbs measure $\mu_n$, $\E$ denotes expectation with respect to the randomness in $\mu_n$, and $\1_A$ denotes the function that is $1$ on the set $A$ and $0$ elsewhere.

It was predicted in a seminal paper of~\citet{mpstv84} that ultrametricity happens because the infinite volume limit of the Gibbs measure can be decomposed into ``hierarchically organized pure states''. Roughly speaking, this means that the configuration space admits a hierarchical clustering, with a number $q_\alpha\in [-1,1]$ attached to each cluster $\alpha$, so that if $\sigma^1$ and $\sigma^2$ are drawn independently from the Gibbs measure, then with high probability, $R_{1,2}\approx q_\alpha$, where $\alpha$ is the smallest cluster containing both $\sigma^1$ and $\sigma^2$ (see Figure \ref{purestate}). Here ``smallest'' means ``lowest down in the hierarchy''.

 
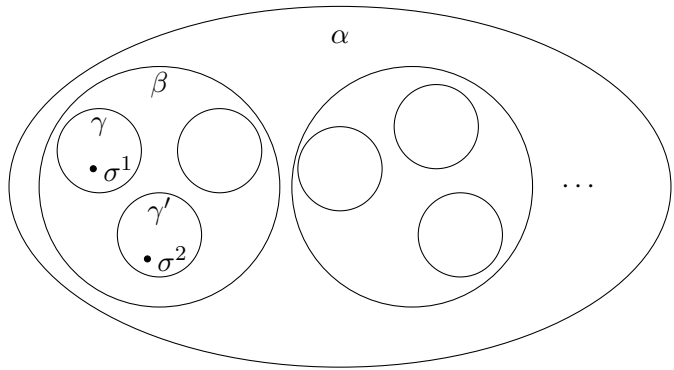
\begin{figure}
\begin{center}
\begin{tikzpicture}[scale = .8]
\draw (0,0) ellipse (5.5 and 3);
\draw (-3,0) circle [radius = 2];
\draw (1.2,0) circle [radius = 2];
\node at (4,0) {$\cdots$};
\draw (-4, .6) circle [radius = .7];
\draw (-2, .6) circle [radius = .7];
\draw (-3, -.8) circle [radius = .7];
\draw (0, .3) circle [radius = .7];
\draw (1.6, 1) circle [radius = .7];
\draw (2, -.8) circle [radius = .7];
\node at (-4, 1) {$\gamma$};
\node at (-3, 1.7) {$\beta$};
\node at (0, 2.5) {$\alpha$};
\node at (-3, -.4) {$\gamma'$};
\draw [fill] (-4.1,.3) circle [radius = 0.05] node [black,right] {$\sigma^1$};
\draw [fill] (-3.2,-1.2) circle [radius = 0.05] node [black,right] {$\sigma^2$};
\end{tikzpicture}
\caption{Hierarchical organization of pure states. Here $\alpha$, $\beta$, $\gamma$ and $\gamma'$ are hierarchically nested clusters representing various pure states, and $\sigma^1\in \gamma$, $\sigma^2\in \gamma'$. But $R_{1,2}\approx q_\beta$, since $\beta$ is the smallest cluster that contains both $\sigma^1$ and $\sigma^2$.\label{purestate}}
\end{center}
\end{figure}

It is not difficult to prove that ultrametricity implies the hierarchical organization of pure states if $R_{1,2}$ can take only finitely many values in the infinite volume limit; this, in fact, is the basis of the heuristic sketched in~\cite{mpstv84}. However, if this condition does not hold --- in which case the system is said to exhibit ``full replica symmetry breaking'' --- then it is not obvious how to establish the hierarchical organization of pure states starting from the Parisi ansatz \eqref{panchenko}.

There are two kinds of systems where the pure state picture has been rigorously established. The first is a class of spin glass models known as pure $p$-spin spherical models, where the pure state construction was given  recently by \citet{subag17}, building on the earlier contributions of~\cite{ab13, abc13, sz17, ac18}. The second is the class of models that have been shown to satisfy the generalized Ghirlanda--Guerra identities. For these models, the construction of pure states was given by \citet{panchenko13} in the infinite volume limit, and recently  by~\citet{jagannath17} in the setting of large but finite $n$. (See also the earlier works of \citet{talagrand06b, talagrand10}.)

Incidentally, the generalized Ghirlanda--Guerra identities are believed to hold in all physically interesting models that satisfy the Parisi ansatz \eqref{panchenko}. Therefore, in principle, the results of \cite{panchenko13, jagannath17} should give the pure state construction in all such models, provided that the identities can be established. However, there are other important models, such as the Sherrington--Kirkpatrick (S-K) model, where it is known that the generalized Ghirlanda--Guerra identities do not hold~\cite[Remark 2.4]{jagannath17}. In the S-K model, it is believed that the absolute value of the overlap, rather than the overlap itself, should satisfy the ultrametric property. To account for such cases, we formulate a generalized version of~\eqref{panchenko}. We will say that a sequence of spin glass models satisfy the {\it generalized Parisi ansatz} if for some bounded measurable $f:[-1,1]\to \R$, 
\begin{align}\label{panchenkogen}
\lim_{n\to\infty} \E\smallavg{\1_{\{f(R_{1,2})\ge \min \{f(R_{1,3}), f(R_{2,3})\} - \ep\}}} = 1
\end{align}
for all $\ep>0$. 
Theorem \ref{mainresult0} allows us to prove that hierarchically organized pure states can be constructed for any system that satisfies this generalized  ansatz. Since the only systems where ultrametricity has been rigorously established are systems where the pure state construction has also been proved, the result gives no immediate gain. But it is intellectually satisfying and potentially useful for the future. For example, if the generalized Parisi ansatz \eqref{panchenkogen} can be proved for the S-K model with $f(x)=|x|$, our theorem will instantly give the construction of pure states. The precise statement is as follows.
\begin{thm}\label{spinthm}
Consider any sequence of spin glass models that satisfy the generalized Parisi ultrametricity ansatz \eqref{panchenkogen} for some bounded measurable function $f$. Then there are  sequences $\ep_n$ and $\delta_n$ tending to zero, such that with probability at least $1-\ep_n$,  the following happens. There is a hierarchical clustering of the configuration space $\Sigma_n$, such that the number of clusters is finite, each cluster is measurable, and for each cluster $\alpha$ there is a number $q_\alpha$ that is a function of its the depth in the hierarchy, with the property that 
\[
\smallavg{|f(R_{1,2})-q_\alpha|}\le \delta_n,
\] 
where $\alpha = \alpha(\sigma^1, \sigma^2)$ is the smallest cluster containing two configurations $\sigma^1$ and $\sigma^2$ drawn independently from the Gibbs measure and $R_{1,2}$ is their overlap. 
\end{thm}

Just for clarity,  we note  that in Theorem \ref{spinthm} the sequences $\ep_n$ and $\delta_n$ are deterministic, but the hierarchical clustering is a function of the Gibbs measure (and hence random). We also note that even though the number of clusters is finite, the number may grow with $n$. Theorem \ref{spinthm} is proved as a simple consequence of Theorem~\ref{mainresult0} in Section~\ref{spinproof}. 

\section{A vertex-weighted regularity lemma}
The key to proving Theorem \ref{mainresult0} is a weighted version of Szemer\'edi's regularity lemma~\cite{szemeredi78}. Although there are a number of weighted regularity lemmas in the literature (such as in \cite{alonetal10, csabapluhar14} and the very recent preprint \cite{gs19}), we could not find the exact version stated below, which is what we needed for proving Theorem \ref{mainresult0}. Therefore a complete proof is given. 

Let $G=(S,E)$ be a finite simple graph. In the following, we will adopt the convention that the set of edges $E$ is the subset of $S^2$ consisting of all $(x,y)$ such that there is an edge between $x$ and $y$. In particular, if there is an edge between $x$ and $y$, then both $(x,y)$ and $(y,x)$ belong to $E$. 

Let $\mu$ be a nonnegative measure on $S$. If $U$ and $V$ are disjoint subsets of $S$, we define the $\mu$-weighted edge-density between $U$ and $V$ as
\[
d(U,V) \coloneqq {\mu^{\otimes 2}((x,y) \in E \colon x \in U, y \in V) \over \mu(U) \mu(V) }.
\]
If the denominator is zero, $d(U,V)$ is undefined. Given $\ep>0$, a pair of disjoint sets $U,V\subset S$ will be called a $\mu$-weighted $\ep$-regular pair if for any $A\subset U$  and $B\subset V$ with $\mu(A)\ge \ep \mu(U)$ and $\mu(B)\ge \ep \mu(V)$, we have
\[
|d(A,B)-d(U,V)|\le \ep. 
\]
The following theorem is a $\mu$-weighted version of Szemer\'edi's regularity lemma. 
\begin{thm}[Vertex-weighted regularity lemma]\label{weightedfinal} Let $G = (S,E)$ a finite simple graph and let $\mu$ be a finite nonnegative measure on $S$. Let 
\[ \mu^* := \max_{x \in S} \mu(x). \]
Take any $\epsilon >0$ and any positive integer $m$. Then there is a positive real number $p(\ep,m)$ and a positive integer $M(\ep, m)$, both depending only on $\ep$ and $m$, such that if $\mu^*\le p(\ep, m)\mu(S)$, then there is a partition $S = V_0 \cup \dots \cup V_q$ with $m \leq q \leq M(\ep,m)$, such that
\begin{enumerate}
\item[\textup{(i)}] $\mu(V_0) \le \epsilon\mu(S)$,
\item[\textup{(ii)}] $\mu(V_i)>0$ and $|\mu(V_i) - \mu(V_j)| \leq \mu^*$ for all $1 \leq i, j \leq q$, and 
\item[\textup{(iii)}] all but at most $\ep q^2$ pairs $(V_i, V_j)$, $1\le i\ne j\le q$, are $\mu$-weighted $\ep$-regular, as defined above.
\end{enumerate}
\end{thm}
The rest of this section is devoted to the proof of this theorem. We follow the spectral approach to proving Szemer\'edi's lemma, pioneered by \citet{KF} and lucidly explained in a blog entry of~\citet{taoblog}. 
If $\mu(S)=0$, there is nothing to prove. So let us assume that $\mu(S)>0$, and  normalize $\mu$ to define a probability measure:
\[
\P(A) := \frac{\mu(A)}{\mu(S)}, \ \ A\subset S. 
\]
Also let
\[
P^* := \max_{x\in S} \P(x) = \frac{\mu^*}{\mu(S)}.
\]
If we prove the theorem for $\P$ instead of $\mu$ (with $P^*$ instead of $\mu^*$), it is easy to see that it proves the theorem for $\mu$. So we will henceforth work with $\P$ instead of $\mu$. 
We will first prove Theorem~\ref{weightedfinal} in the case that $\P(x)$ is rational for all $x \in S$.
\begin{lmm}\label{weightedszem} The vertex-weighted regularity lemma  holds if $\P(x)$ is rational for each $x$. 
\end{lmm}
\begin{proof}
Note that if $\ep<\ep'$, then an $\ep$-regular partition is also an $\ep'$-regular partition. So let us assume without loss of generality that $\ep<1/4$.

Since $\P(x)$ is rational for every $x$, we can find an integer $N$ such that $K(x) := N \P(x)$ is an integer for every $x$.  
Let $[N]:= \{1,\ldots, N\}$.  Choose a map $f \colon [N] \to S$ such that $|f^{-1}(x)| = K(x)$ for every $x$, and these inverse images are disjoint. (This is possible is $\P(S)=1$.)  
Let $G_N = ([N],E_N)$ be a graph with vertices $[N]$, and $(x,y) \in E_N$ if and only if $(f(x),f(y)) \in E$.  

Let $H$ be the adjacency matrix of $G_N$.  Then $H$ has a spectral decomposition
\[ 
H = \sum_{i=1}^N \lambda_i u_iu_i^T, 
\]
where $u_i^T$ denotes the transpose of the column vector $u_i$. We will assume the $\lambda_i$'s are numbered in order of decreasing magnitude, that is,
\begin{equation}\label{lambdadec}
|\lambda_1| \geq |\lambda_2 | \geq \dots \geq |\lambda_N |. 
\end{equation}
Let $F:\Z_+\to \R_+$ be a function satisfying $F(j) > j$ for all $j$. The exact choice of $F$ will be made later, and it will depend on $\epsilon$ and $m$  (but not on anything else).  Partition the set $\{1,\ldots,N\}$ into sets of the form $\{i: z_k \le i< z_{k+1}\}$, where  $z_0=1$ and for $k\ge 1$, 
\[
z_k = \underbrace{F\circ F\circ\cdots \circ F}_{k \text{ times}}(1).
\]
Note that since $F(j)> j$ for all $j$, $z_k$ is a strictly increasing sequence. Also, since
\[ 
\tr(H^2) = \sum_{i=1}^N \lambda_i^2 = 2|E_N| \leq N^2, 
\]
there exists $k\le 128\ep^{-5}+1$ such that
\[
\sum_{z_k\le i< z_{k+1}} \lambda_i^2 \leq \frac{\epsilon^5 N^2}{128}.
\]
Consequently, there exists an integer $J$ such that $J$ is bounded by a constant that depends only on $\epsilon$ and $m$, and 
\begin{equation} 
\sum_{J\le i< F(J)} \lambda_i^2 \leq \frac{\epsilon^5 N^2}{128}. \label{eq:T2small}
\end{equation}
If $\lambda_J\ne 0$, then by \eqref{lambdadec}, $\lambda_i\ne 0$ for all $i<J$. If $\lambda_J = 0$, then again by \eqref{lambdadec}, there is some $J'\le J$ such that $\lambda_i \ne 0$ for all $i<J'$ and $\lambda_i=0$ for all $i\ge J'$. Thus, by decreasing $J$ if necessary, we can ensure that $\lambda_i\ne 0$ for all $i<J$. Henceforth, we will assume that this holds. Let
\[ 
H_1 = \sum_{i<J} \lambda_i u_i u_i^T, \quad H_2 = \sum_{J\le i< F(J)} \lambda_i u_i u_i^T, \quad H_3 = \sum_{i\ge F(J)} \lambda_i u_iu_i^T.
\]
Then the number of edges $E_N(A,B)$ between sets $A,B \subset [N]$ is
\[ E_N(A,B) = \1_A^T H_1 \1_B + \1_A^T H_2 \1_B + \1_A^T H_3 \1_B \]
where $\1_A$ is the vector that has $1$ at the coordinates that belong to $A$ and $0$ elsewhere.  
For each $i < J$, define
\[ 
W^{(i)}_0 = \left\{ y \in [N] \colon |u_i(y)| > \sqrt{\frac{2J}{\epsilon N}} \right\},
\]
where $u_i(y)$ denotes the $y^{\text{th}}$ coordinate of $u_i$. 
Then, since $u_i$ is a unit vector,
\begin{align*}
1 &= \sum_{y \in [N]} u_i(y)^2 \geq \sum_{y \in W^{(i)}_0} u_i(y)^2 \geq \frac{2J}{\epsilon N} |W_0^{(i)}|,
\end{align*}
so that $|W_0^{(i)}| \leq \epsilon N/2J$.  Thus if 
\[
W_0 := \bigcup_{i<J} W_0^{(i)},
\]
then $|W_0| \le \epsilon N/2$.  Now partition $[N] \setminus W_0^{(i)}$ as the union of $\{ W_{k}^{(i)}: |k| \leq 32J^2/\epsilon^2+1\}$, where 
\[ 
W_{k}^{(i)} = \left\{ y \in [N] \setminus W_0^{(i)}\colon u_i(y) \in \frac{\epsilon^{3/2}}{16 \sqrt{2J^3 N}} (k-1,k] \right\}. 
\]
After doing this for $i = 1,\dots, J-1$, set
\[ 
W_{k_1,\dots,k_{J-1}} = \bigcap_{i<J} W_{k_i}^{(i)}.
\]
Note that $\{ W_{k_1,\dots,k_{J-1}} \}$ is a partition of $[N] \setminus W_0$. Enumerate the partition sets as $W_1,\dots,W_r$. From the definition of the partition, it is clear that
\begin{equation}\label{rbd}
r \le \biggl(\frac{64J^2}{\epsilon^2} + 3\biggr)^J.
\end{equation}
We will use this bound on $r$ later. Now, since $H$ is the adjacency matrix of a graph on $N$ vertices, a standard result from linear algebra implies that $|\lambda_1|\le N$. Thus, for $x,y \in W_{k_1,\dots,k_{J-1}}$ and $w,z \in W_{k_1',\dots,k_{J-1}'}$,
\begin{align*}
&|\1_w^T H_1 \1_x - \1_z^T H_1 \1_y| = \left| \sum_{i<J} \lambda_i\left( u_i(w)u_i(x) - u_i(z) u_i(y) \right)  \right| \\
&\leq |\lambda_1| \sum_{i<J} \left( \left| (u_i(w)  - u_i(z)) u_i(x) \right| + \left| u_i(z)( u_i(x) - u_i(y)) \right| \right) \\
&\leq 2N \sum_{i<J} \sqrt{\frac{2J}{\epsilon N}} \left( \frac{\epsilon^{3/2}}{16\sqrt{2J^3N}} \right) \le \frac{\epsilon}{8}.
\end{align*}
For $1\le i,j\le r$, define
\begin{align}\label{dijdef}
d_{ij} := \frac{1}{|W_i||W_j|} \sum_{x \in W_i,y \in W_j} \1_x^T H_1 \1_y.
\end{align}
Then for any $A \subset W_i$ and $B \subset W_j$, the above inequality shows that
\begin{align}
\left|\1_A^T H_1 \1_B -d_{ij} |A| |B|\right| &= \biggl|\sum_{w \in A, x \in B} \1_w^T H_1 \1_x - d_{ij} |A| |B|\biggr|\nonumber\\
&= \biggl|\frac{1}{|W_i||W_j|}\sum_{\substack{w \in A, x \in B\\ z\in W_i, y\in W_j}} (\1_w^T H_1 \1_x - \1_z^T H_1 \1_y)\biggr|\nonumber\\
&\leq \frac{1}{|W_i||W_j|}\sum_{\substack{w \in A, x \in B\\ z\in W_i, y\in W_j}} |\1_w^T H_1 \1_x - \1_z^T H_1 \1_y|\nonumber\\
&\leq \frac{\epsilon}{8} |A|B|.
\label{eq:noflux}
\end{align} 
We will use this inequality later. We now claim that each $W_j$, $0\le j\le r$, is the pre-image of some subset of $S$ under the map $f$. To see this, first note that if $f(x) = f(y)$, then clearly $H\1_x = H\1_y$.  In terms of the spectral decomposition, this can be written as 
\[ \sum_{i=1}^N \lambda_i u_i(x) u_i = \sum_{i=1}^N \lambda_i u_i(y) u_i. \]
By the linear independence of the $u_i$'s, this shows that for each $i$, $\lambda_i = 0$ or $u_i(x) = u_i(y)$.  But if $i < J$, then $\lambda_i \neq 0$, and so $x$ and $y$ must belong to the same $W_k^{(i)}$. Since this holds for all $i<J$, $x$ and $y$ belong to the same $W_j$.

Next, we make the partition equitable by subdividing the $W_j$'s.  By what we just showed, $W_j$ is the union of $f^{-1}(x)$ for some set of $x\in S$. Note that for each $x$, the pre-image $|f^{-1}(x)|$ has size at most $P^* N$.  Let
\[ 
m^* = {m \over 1 - P^*m}. 
\]
If $P^*$ is sufficiently small (depending on $m$), $m^*$ is positive.  Partition $W_j$ by sorting the pre-images into subsets of size as close as possible to $\epsilon N/2(r+m^*)$ but no smaller, and one remainder set of size less than $\epsilon N/2(r+m^*)$.  So, 
\[ W_j = U_0^{(j)}\cup \biggl(\bigcup_{k\ge 1} U_k^{(j)}\biggr) \]
with
\[ |U_0^{(j)}| < {\epsilon N \over 2(r+m^*)} \]
and for $k \geq 1$,
\begin{equation}\label{eq:vkbounds}
{\epsilon N \over 2(r+m^*)} \leq |U_k^{(j)}| \leq \left( {\epsilon \over 2(r+m^*)} + P^* \right)N.
\end{equation}
The union of the remainder sets is small:
\[ \biggl|\bigcup_{j=1}^r U_0^{(j)}\biggr| \le {\epsilon r N \over 2(r+m^*)} \le  \frac{\epsilon N}{2}. \]
Define
\[ U_0 = W_0 \cup \biggl( \bigcup_{j=1}^r U_0^{(j)} \biggr) \]
as the exceptional set, and relabel the remaining partition sets $\{ U_k^{(j)} \}_{k,j}$ as $U_1,\dots,U_q$.  Then $|U_0|\le \ep N$, and hence by \eqref{eq:vkbounds}, 
\begin{equation} \frac{1-\epsilon}{\epsilon/2(r+m^*) + P^*} \leq q \leq {2(r+m^*) \over \epsilon}. \label{eq:qbds} \end{equation}
Since $r$ can be bounded by a quantity that depends only on $\epsilon$ and $m$, we can let $M(\epsilon,m)$ to be an upper bound, depending only on $m$ and $\ep$, for the quantity $2(r+m^*)/\ep$. 
Now notice that
\[  {1 - \epsilon \over \epsilon/2(r+m^*) + P^*} \geq {1-\epsilon \over \epsilon/2m^* + P^*}.   \]
Using the definition of $m^*$, we have
\begin{align*}
 {1-\epsilon \over \epsilon/2m^* + P^*} &= \biggl({1-\epsilon \over \epsilon + (2-\epsilon)P^* m}\biggr) 2m
\end{align*}
Thus,  sufficient smallness of $P^*$ (depending on $m$ and $\ep$) ensures that $q \geq m$.

By construction of $U_0,\ldots,U_q$, there is a partition $V_0,\ldots, V_q$ of $S$ such that $U_i = f^{-1}(V_i)$ for each $i$.  Note that
\[ \P(V_0) = \frac{1}{N} |U_0| \le \epsilon, \]
and for $i \geq 1$,
\begin{equation}\label{vbds}
 \frac{\epsilon}{2(r+m^*)} \leq \P(V_i) \leq \frac{\epsilon}{2(r+m^*)} + P^*, 
\end{equation}
which implies, in particular, that $|\P(V_i)-\P(V_j)|\le P^*$ for all $1\le i,j\le q$.  This also shows that $\P(V_i)>0$ for all $1\le i\le q$. 

Next, note that by \eqref{eq:T2small}, $\tr(H_2^2) \leq \epsilon^5 N^2/128$.  Thus if $H_2 = [x_{ab}]_{a,b=1}^N$, then 
\begin{align}
\frac{\epsilon^5 N^2}{128} &\geq \sum_{a,b=1}^N x_{ab}^2. \label{eq:xabbd}
\end{align}
Let $X_{ij} = \sum_{a \in U_i,b \in U_j} x_{ab}^2$, and let 
\[
 \Sigma := \biggl\{(i,j): X_{ij} > \frac{\epsilon^4}{64} |U_i||U_j| \biggr\}.
\]
Let $\nu$ be the measure on $\{1,\dots,q\}^2$ such that $\nu(i,j) = |U_i||U_j|$ for each $i$ and $j$. 
Then 
\begin{align*}
\nu\left( \Sigma \right) &= \sum_{(i,j)\in \Sigma} |U_i||U_j|\\
&\leq \frac{64}{\epsilon^4} \sum_{i,j=1}^q X_{ij} = \frac{64}{\ep^4}\sum_{i,j=1}^q \sum_{a \in U_i,b \in U_j} x_{ab}^2 \leq \frac{64}{\ep^4}\sum_{a,b = 1}^N x_{ab}^2.
\end{align*}
Thus, by~\eqref{eq:xabbd}, $\nu(\Sigma) \leq \epsilon N^2/2$.  We can use this to bound $|\Sigma|$, as follows.  By the inequalities~\eqref{eq:vkbounds} and~\eqref{eq:qbds},
\begin{align*}
\frac{1}{|U_i|} &\leq \frac{2(r+m^*)}{\ep N}\\
&\leq \frac{2(r+m^*)}{\ep N} \biggl(\frac{(\epsilon/2(r+m^*) + P^*)q}{1-\epsilon}\biggr) \\
&= \biggl( \frac{\epsilon+ 2P^*(r+m^*)}{\epsilon (1-\epsilon)} \biggr) \frac{q}{N}.
\end{align*}
Thus,
\begin{align*}
| \Sigma | &= \sum_{(i,j) \in \Sigma} { \nu(i,j) \over |U_i||U_j| } \\
&\leq \nu(\Sigma) \biggl( \frac{\epsilon + 2P^*(r+m^*)}{\epsilon (1-\epsilon)} \biggr)^2 \frac{q^2}{N^2} \leq \frac{\epsilon}{2} \biggl( \frac{\epsilon + 2P^*(r+m^*)}{\epsilon (1-\epsilon)} \biggr)^2 q^2.
\end{align*}
Recall that $r$ is bounded by a constant that depends only on $\ep$ and $m$, and that $\ep<1/4$. Thus, if $P^*$ is sufficiently small (depending on $\ep$ and $m$), this gives 
\[
|\Sigma|\le \ep q^2.
\]
Suppose that $(i,j) \notin \Sigma$.  Then for $Q \subset U_i$ and $R \subset U_j$ with $|Q| \geq \epsilon |U_i|$ and $|R| \geq \epsilon |U_j|$, the Cauchy--Schwarz inequality and the definition of $\Sigma$ imply  that
\begin{align}
 |\1_Q^T H_2 \1_R| &\le \sum_{a\in Q, b\in R} |x_{ab}|\notag\\
 &\le \sqrt{|Q||R|} \biggl(\sum_{a\in Q, b\in R} x_{ab}^2\biggr)^{1/2}\notag\\
 &\le \sqrt{|Q||R|}\biggl(\sum_{a \in U_i,b \in U_j} x_{ab}^2\biggr)^{1/2}\notag\\
 &\leq \frac{\ep^2}{8}\sqrt{|Q||R||U_i||U_j|} \le 
 \frac{\epsilon}{8} |Q||R|. \label{h2bd}
\end{align}
Next, note that for any choice of $(i,j) \in \{1,\dots,q\}^2$, and for any $Q\subset U_i$ and $R\subset U_j$, 
\begin{align*}
\1_Q^T H_3 \1_R &= \sum_{k\ge F(J)} \lambda_k \1_Q^Tu_k u_k^T\1_R.
\end{align*}
Since $\sum_{k=1}^N \lambda_k^2 \leq N^2$, and the $\lambda_k$ are in order of decreasing magnitude, we have
\[ N^2 \geq k\lambda_k^2, \]
so that $|\lambda_k| \leq N/\sqrt{k}$.  Thus,
\begin{align}
 |\1_Q^T H_3 \1_R| &\leq \frac{N}{\sqrt{F(J)}}  \sum_{k\ge F(J)} |\1_Q^Tu_k u_k^T\1_R| \notag\\
 &\leq \frac{N}{\sqrt{F(J)}} \|\1_Q\|\|\1_R\|\notag\\
 &=  \frac{N}{\sqrt{F(J)}}  \sqrt{|Q||R|}.\label{h3bd}
\end{align}
Now take any $1\le i,j\le q$. Let $k$ and $l$ be indices such that $U_i\subset W_{k}$ and $U_j\subset W_{l}$. Define $\delta_{ij}:= d_{kl}$, where $d_{kl}$ is the quantity defined in \eqref{dijdef}. 
Then by \eqref{eq:noflux}, \eqref{h2bd} and \eqref{h3bd}, we see that if  $Q \subset U_i$ and $R \subset U_j$, with $(i,j) \in \{1,\dots,q\}^2 \setminus \Sigma$, and $|Q| \geq \epsilon |U_i|$ and $|R| \geq \epsilon |U_j|$, then 
\begin{align*}
|\1_Q^T H \1_R- \delta_{ij} |Q| |R|| &\le |\1_Q^T H \1_R - \1_Q^T H_1 \1_R|+\frac{\ep}{8} |Q||R|\\
&\le |\1_Q^T H_2 \1_R|  + |\1_Q^T H_3 \1_R| +\frac{\ep}{8}|Q||R|\\
&\le  \frac{\ep}{4} |Q||R| + \frac{N}{\sqrt{F(J)}} \sqrt{|Q||R|}.
\end{align*}
Now take any $(i,j)\in \{1,\ldots, q\}^2\setminus \Sigma$, and any $A\subset V_i$ and $B\subset V_j$ with $\P(A)\ge \ep \P(V_i)$ and $\P(B) \ge \ep \P(V_j)$. Let $Q := f^{-1}(A)$ and $R:= f^{-1}(B)$. Then $Q\subset U_i$, $R\subset U_j$, $|Q|\ge \ep |U_i|$ and $|R|\ge \ep |U_j|$. Also,
\begin{align*}
\1_Q^T H \1_R &= N^2 \P(A) \P(B) d(A, B),
\end{align*}
and $|Q||R| = N^2 \P(A)\P(B)$.  Thus, the above calculations show that 
\begin{align*}
|\1_Q^T H \1_R-\delta_{ij} N^2 \P(A) \P(B)| &= |\1_Q^T H \1_R- \delta_{ij} |Q| |R||\\
&\le  \frac{\epsilon}{4}|Q||R| + \frac{N}{\sqrt{F(J)}} \sqrt{|Q||R|} \\
&= \frac{\epsilon}{4} N^2 \P(A) \P(B) + \frac{N^2}{\sqrt{F(J)}}  \sqrt{\P(A)\P(B)}.
\end{align*}
Combining the last two displays and dividing throughout by $N^2 \P(A) \P(B)$, we get
\begin{align*}
|d(A, B) - \delta_{ij}|\le  \frac{\epsilon}{4}+  \frac{1}{\sqrt{F(J)\P(A)\P(B)}}.
\end{align*}
Recalling that $\P(A) \geq \epsilon \P(V_i)$ and $\P(B) \geq \epsilon \P(V_j)$, and applying \eqref{vbds}, we get
\[ \frac{1}{\sqrt{\P(A)\P(B)}} \leq \frac{1}{\ep \sqrt{ \P(V_i) \P(V_j)}} \leq \frac{2(r+m^*)}{\epsilon^2}. \]
Now suppose $F$ is chosen in such a way that we can guarantee 
\begin{equation}\label{fcrit}
\frac{1}{\sqrt{F(J)}} \biggl( \frac{2(r+m^*)}{\epsilon^{2}} \biggr) \le \frac{\epsilon}{4}. 
\end{equation}
Then from the above bounds it will follow that 
\[
|d(A,B)-\delta_{ij}|\le \frac{\epsilon}{2}.
\]
Replacing $A$ be $V_i$ and $B$ by $V_j$, we also have $|d(V_i, V_j)-\delta_{ij}|\le \epsilon/2$. Thus, we would get
\[
|d(A,B)-d(V_i,V_j)|\le \epsilon,
\]
which would complete the proof. So we only have to guarantee \eqref{fcrit}. By the bound on $r$ from \eqref{rbd}, we see that \eqref{fcrit} holds if 
\[ F(J) \geq \frac{\left(8(64J^2/\epsilon^2+3)^{J}+8m/(1 - P^*m) \right)^2}{\epsilon^{6}}. \]
Assuming that $P^*\le 1/2m$, it is now easy to choose $F$, depending only on $\ep$ and $m$, satisfying the above criterion for every $J\in \Z_+$. 
\end{proof}
In the final step, we now drop the rationality assumption and prove Theorem \ref{weightedfinal}. 
\begin{proof}[Proof of Theorem \ref{weightedfinal}]
Enumerate $S= \{x_1,\ldots,x_n\}$ and let $p_i := \P(x_i)$.  Take any positive real number $\nu$. Let $q_1,\ldots, q_n$ be positive rational numbers such that $p_i\le q_i \le p_i+\nu$ for each $i$. Let $r_i := q_i/\sum q_j$, so that $r_1,\ldots, r_n$ are again rational, $\sum r_i = 1$, and for each $i$, 
\begin{align*}
|p_i - r_i| &\le |p_i-q_i|+ |q_i-r_i|\\
&\le \nu + q_i \left|1-\frac{1}{\sum q_j}\right|\\
&\le \nu + (1+\nu)\frac{\sum|q_j-p_j|}{\sum q_j} \\
&\le \nu +(1+\nu)\sum |q_j-p_j|\le \nu+n(1+\nu)\nu. 
\end{align*}
Define the modified weight $\P^{(\nu)}(x_i) :=  r_i$. Suppose that $P^*\le \frac{1}{2}p(\ep,m)$, where $p(\ep,m)$ is the bound on the maximum atom required in Lemma \ref{weightedszem}. Then for sufficiently small $\nu$, the above display shows that we can apply Lemma~\ref{weightedszem} to $\P^{(\nu)}$. Suppose that we get an $\ep$-regular partition $V_0^{(\nu)},\ldots, V_q^{(\nu)}$ of $S$. 
Now let $\nu \to 0$. We get a partition as above for each $\nu$. Since the number of possible partitions is finite, there is a subsequence along which the partitions stabilize for sufficiently small $\nu$. This allows us to define a limiting partition along this subsequence. Since $\P^{(\nu)}(x)\to \P(x)$ for every $x$ (by the above display), is straightforward to verify that this limiting partition is $\ep$-regular for $\P$. 
\end{proof}

\section{Preliminary steps}
\label{setup}
In this section we begin the steps towards the proof of Theorem \ref{mainresult0}. First, note that by rescaling $s$ if necessary, we may assume that $b=1$.  We will work under this assumption for the rest of the paper.

Right away, we begin by observing that the converse statement in Theorem \ref{mainresult} is very easy to prove: Take any $\delta>0$. Suppose that 
\begin{align*}
\textup{Tree}(S,\mf, \P, s)< \delta. 
\end{align*}
Then there exists a tree $T$ with root $r$, finite diameter, and set of leaves $S$, and some $\alpha\ge 0$, such that $(X,Y)_r$ is a measurable random variable and 
\[
\E|s(X,Y) - \alpha(X,Y)_r|<\delta,
\]
where $X$ and $Y$ are i.i.d.~draws from $\P$. By Markov's inequality,
\begin{align*}
\P(|s(X,Y) - \alpha(X,Y)_r|\ge \sqrt{\delta}) \le \sqrt{\delta}. 
\end{align*}
Therefore if $X$, $Y$ and $Z$ are i.i.d.~draws from $\P$, then with probability at least $1-3\sqrt{\delta}$, the quantities $|s(X,Z) - \alpha(X,Z)_r|$,  $|s(Y,Z) - \alpha(Y,Z)_r|$ and $|s(X,Y) - \alpha(X,Y)_r|$ are all bounded above by $\sqrt{\delta}$. If this happens, then
\begin{align*}
&\min\{s(X,Z), s(Y,Z)\} - s(X,Y) \\
&\le \min\{\alpha(X,Z)_r, \alpha (Y,Z)_r\} - \alpha(X,Y)_r + 2\sqrt{\delta}\\
&= \alpha (\min\{(X,Z)_r,(Y,Z)_r\} - (X,Y)_r) + 2\sqrt{\delta}.
\end{align*}
Now, since $(x,y)_r$ is a Gromov product under the graph distance on a tree, it satisfies 
\[
(x,y)_r \ge \min\{(x,z)_r, (y,z)_r\}
\]
for all $x,y,z$. Thus, we get
\[
\min\{s(X,Z), s(Y,Z)\} - s(X,Y)\le 2\sqrt{\delta}. 
\]
Recall that this happens with probability at least $1-3\sqrt{\delta}$. Also, we have assumed that $b=1$. Thus,
\begin{align*}
\textup{Hyp}(S,\mf, \P,s) &= \E(\min\{s(X,Z), s(Y,Z)\} - s(X,Y))_+\\
&\le 2\sqrt{\delta} + 3\sqrt{\delta} = 5\sqrt{\delta}. 
\end{align*}
This proves the converse part of Theorem \ref{mainresult0}. 

We now start our journey towards the proof of the main assertion of Theorem \ref{mainresult0}, namely, that  if  $\textup{Hyp}(S,\mf, \P, s)$ is small, then $\textup{Tree}(S,\mf,\P,s)$ is also small. We will first prove the following weaker theorem. At the very end of the paper, we will complete the proof of Theorem \ref{mainresult0} using this theorem. 
\begin{thm}\label{mainresult}
Assume that $S$ is a finite set, $\mf$ is the power set of $S$, $\P$ is a probability measure defined on $\mf$, and $s:S\times S\to [0,1]$ is a symmetric function. Let $P^*:= \max_{x\in S} \P(x)$. Then given any $\ep>0$, there is some $\delta>0$  depending only on $\ep$, such that if $P^*< \delta$ and $\textup{Hyp}(S, \mf, \P, s)< \delta$, then $\textup{Tree}(S,\mf,\P, s)< \ep$. 
\end{thm}
From here until the end of the proof of Theorem \ref{mainresult}, we will work under the assumptions stated above. Take any $\delta>0$ and suppose that 
\[
\textup{Hyp}(S,\mf,\P,s)<\delta.
\]
A basic step is to show that for most values of $t \in [0,1]$, the set
\begin{align}\label{rtdef}
 R_t \coloneqq \{ (x,y,z) \colon s(x,y) < t \leq \min\{ s(x,z),s(y,z)\} \} \end{align}
has small probability.  For convenience, let 
\[
\delta_0 := \delta^{1/8}.
\]
The above definition of $\delta_0$ will be fixed throughout the remainder of the proof. 
\begin{lmm}
Let $R := \{ t \colon \P^{\otimes 3}(R_t) \geq \delta_0^4 \}$. Then $\mathscr{L}(R) \leq \delta_0^4$, where $\mathscr{L}$ is Lebesgue measure.
\label{Rsmall}
\end{lmm}
\begin{proof}
Define
\[ \mathscr{R}(x,y,z) = \{ r\in [0,1] \colon s(x,y) < r \leq \min\{s(x,z),s(y,z)\} \}. \]
Note that
\[ \P^{\otimes 3}(R_t) = \sum_{x,y,z \in S} \P^{\otimes 3}(x,y,z) \1_{\mathscr{R}(x,y,z)}(t). \]
Thus, 
\begin{align*}
\int_{0}^1 \P^{\otimes 3}(R_t) \d t &= \sum_{(x,y,z) \in S^3} \int_0^1 \P^{\otimes 3}(x,y,z) \1_{\mathscr{R}(x,y,z)}(t) \d t \\
&= \sum_{(x,y,z) \in S^3} \P^{\otimes 3}(x,y,z) (\min\{s(x,z),s(y,z) \} - s(x,y))_+ \\
&= \textup{Hyp}(S, \mf,\P, s)\le \delta = \delta_0^8.
\end{align*}
If $\mathscr{L}$ is Lebesgue measure on $[0,1]$, the definition of $R$ implies that
\[ \int_{0}^1 \P^{\otimes 3}(R_t) \d t \geq \delta_0^4 \mathscr{L}(R). \]
The claimed result now follows easily by combining the two displays. 
\end{proof}

Let us now fix some $\ep\in (0,1)$ and $m\ge 2$. This $\ep$ and $m$ will remain fixed throughout the rest of the proof. At various steps, we will need to assume that $\ep$ is smaller than some universal constant (such as $\ep<1/9$) or $m$ is bigger than some universal constant (such as $m\ge 20$), and we will make these assumptions without explicitly stating so. 

Having chosen $\ep$ and $m$, define  
\begin{equation}\label{kappadef}
\kappa := \max\{\ep^{1/24}, m^{-1/2}\}.
\end{equation}
Assume that $\delta_0< \kappa/2$.  Let $N$ be the largest integer such that $N\kappa < 1$. Note that $N\le 1/\kappa\le 1/\delta_0$. In particular, $N$ is bounded by a constant that depends only on $\ep$ and $m$. We will use this information later. 
By Lemma~\ref{Rsmall}, any subinterval of $[0,1]$ of length $\ge \delta_0$ intersects $R^c$. Thus, we can find a sequence $0 < t_1 < t_2 < \dots < t_N < 1$
such that for each $i$, $t_i \in R^c$ and 
\begin{equation}\label{tidef}
| t_i - i\kappa | \le \delta_0.
\end{equation}
For $y,z\in S$ and $i \in \{1,\dots, N\}$, define three sets:
\begin{align*}
&\frak{R}^1(y,z) \coloneqq \bigcup_{i=1}^N \{ x \in S \colon s(x,y) < t_i \leq \min\{ s(x,z),s(y,z) \} \}, \\
&\frak{R}^2(z) \coloneqq \bigcup_{i=1}^N \{ (x,y) \in S^2 \colon s(x,y) < t_i \leq \min\{ s(x,z),s(y,z) \}, \\
&B(z) \coloneqq \{ y \in S \colon \P( \frak{R}^1(y,z) ) > \delta_0 \}.
\end{align*}
Finally, let 
\[ A := \{ z \colon \P(B(z)) > \delta_0 \}. \]
We now prove two lemmas that will be used several times in the sequel.
\begin{lmm}\label{asmall}
Let $A$ be the set defined above. Then $\P(A)\le \delta_0$. 
\end{lmm}
\begin{proof}
By the choice of $t_i$,
$\P^{\otimes 3}(R_{t_i}) \le \delta_0^4$ for every $i$.
Since $N\le 1/\delta_0$, this gives
\begin{align*}
\P^{\otimes 3}\biggl( \bigcup_{i=1}^N R_{t_i} \biggr) \leq \delta_0^3.
\end{align*}
Thus
\begin{align*}
\delta_0^3 &\geq \sum_{z \in A } \P(z) \P^{\otimes 2}( (x,y) \colon (x,y,z)\in R_{t_i} \text{ for some $i$}  ) \\
&\geq \sum_{z \in A} \P(z)\biggl( \sum_{y \in B(z)} \P(y) \P(\frak{R}^1(y,z) )\biggr) \\
&\ge \sum_{z \in A} \P(z) \P(B(z)) \delta_0 \ge \P(A) \delta_0^2,
\end{align*}
which gives $\P(A) \le \delta_0$. 
\end{proof}
\begin{lmm}\label{r2small}
If $z\notin A$, then $\P^{\otimes 2}(\frak{R}^2(z))\le 2\delta_0$. 
\end{lmm}
\begin{proof}
By the definition of $B(z)$,  
\begin{align*}
\P^{\otimes 2}(\frak{R}^2(z)) 
&= \sum_{y \in B(z)} \P(y)\P(\frak{R}^1(y,z))  + \sum_{y \notin B(z)} \P(y) \P( \frak{R}^1(y,z)) \\
&\leq \P(B(z)) + \delta_0.
\end{align*}
On the other hand, since $z\notin A$, $\P(B(z))\le \delta_0$. This completes the proof.
\end{proof}

\section{Formation of approximate cliques}
\label{modifyingG}
In this section we carry out the main step in the proof of Theorem \ref{mainresult}. We continue with the notations introduced in the previous section. In particular, $P^*$, $\delta_0$, $R$, $R_t$, $\frak{R}^1(y,z)$, $\frak{R}^2(z)$, $B(z)$, $A$, $\ep$, $m$, $\kappa$, $N$ and $t_1,\ldots, t_N$ remain the same as before. 

Take any nonempty set $S'\subset S\setminus A$. Take any $t\in \{t_1,\ldots, t_N\}$, and put an edge between $x,y\in S'$ if and only if $s(x,y)\ge t$. Let $E$ denote this set of edges, and let $G$ be the graph $(S',E)$. Let us continue to denote the restriction of $\P$ to $S'$ by $\P$. Note that this restriction is a measure on $S'$, but not necessarily a probability measure.  

Let $p(\epsilon, m)$ and $M(\epsilon, m)$ be as in Theorem~\ref{weightedfinal}. Throughout this section, we will assume that $\P(S')$ is sufficiently large in comparison to $P^*$ so that 
\begin{equation}\label{pps}
P^*\le \min\biggl\{p(\ep,m), \frac{1}{4M(\ep,m)}\biggr\}\P(S').
\end{equation} 
A first consequence of this assumption is that we can apply Theorem \ref{weightedfinal} to get a partition  $V_0,\ldots, V_q$ of $S'$ with the required properties. For $B', B\subset S'$, let
\[
\rho(B',B) := \P^{\otimes 2}((x,y)\in E: x \in B', y \in B),
\]
so that in the notation of Theorem \ref{weightedfinal},
\[
d(B', B) = \frac{\rho(B', B)}{\P(B)\P(B')}. 
\]
We will fix all of the above throughout the rest of this section. The main result of the section is that $G$ can be slightly modified to make it a disjoint union of cliques. We arrive at this result in several steps. First, we show that $\P(V_i)$ is appropriately close to $\P(S')/q$. 
\begin{lmm}\label{vlmm}
For each $1\le i\le q$, 
\[
\biggl|\P(V_i)-\frac{\P(S')}{q}\biggr|\le \frac{\P(S')}{2q}.
\]
In particular, $\P(V_i)\ge C(\ep,m) \P(S')$, where $C(\ep,m)$ is a positive real number that depends only on $\ep$ and $m$. 
\end{lmm}
\begin{proof}
By construction, $|\P(V_i)-\P(V_j)|\le P^*$ for all $1\le i,j\le q$. Thus, for any $1\le i\le q$,
\begin{align*}
\P(V_i) &\ge \frac{1}{q}\sum_{j=1}^q (\P(V_j)-P^*) \\
&= \frac{\P(S')-\P(V_0)}{q} - P^* \ge \frac{(1-\ep)\P(S')}{q}-P^*\\
&\ge \biggl(\frac{1-\ep}{q}-\frac{1}{4M(\ep, m)}\biggr)\P(S'),
\end{align*}
where the last inequality follows from \eqref{pps}. Similarly,
\begin{align*}
\P(V_i) &\le \frac{1}{q}\sum_{j=1}^q (\P(V_j)+P^*) \le \frac{\P(S')}{q}+P^*\\
&\le \biggl(\frac{1}{q}+\frac{1}{4M(\ep, m)}\biggr)\P(S'),
\end{align*}
Assume that  $\ep<1/4$ (which we can, by our stated convention that $\ep$ can be taken to be less than any universal constant). Since $q\le M(\ep,m)$, this completes the proof. 
\end{proof}
Next, we prove two key lemmas. The first one shows that for any regular pair $(V_i,V_j)$, $d(V_i,V_j)$ is either close to zero or close to one. 
\begin{lmm}\label{mainlmm}
There exists a number $\delta^*$ depending only on $\ep$, $m$ and $\P(S')$, such that if $\delta_0\le \delta^*$, then the following holds. If $(V_i, V_j)$ is an $\epsilon$-regular pair, and $d(V_i, V_j)\ge 3\ep$, then $d(V_i, V_j)\ge 1-2\ep$. 
\label{clusters}
\end{lmm}
The plan of the proof is roughly as follows (see Figure \ref{nnprime} for a schematic representation). We will first find some $x_0\in V_i$ that connects to a substantial fraction of points in $V_j$, where ``substantial'' means a set of $\P$-measure greater than $C\ep\P(V_j)$ for some universal constant $C$. Call this set $N_j(x_0)$. By regularity, the edge density between $N_j(x_0)$ and $V_i$ will be substantial. This will allow us to find $y_0\in N_j(x_0)$ which connects to a substantial fraction of points in $V_i$. Call this set $N_i(y_0)$. Now take any $b\in N_i(y_0)$ and $a\in N_j(x_0)$. Since $x_0$ is a neighbor of $y_0$ and $x_0$ is also a neighbor of $a$, the small hyperbolicity of $S$ will allow us to conclude that it is highly likely that $a$ is a neighbor of $y_0$. But if that happens, then since $b$ is a neighbor of $y_0$ and $a$ is also a neighbor of $y_0$, it is highly likely that $b$ is a neighbor of $a$.  From this, we will conclude that the edge density between $N_j(x_0)$ and $N_i(y_0)$ is close to $1$. Since these sets have substantial size, regularity of $(V_i,V_j)$ will imply that $d(V_i, V_j)$ is close to $1$.

\begin{figure}
\begin{center}
\begin{tikzpicture}
\draw (-3,0) ellipse (2 and 3);
\draw (3,0) ellipse (2 and 3);
\node at (-3,2) {$V_i$};
\node at (3,2) {$V_j$};
\draw [lightgray, fill] (-3,-1) ellipse (1.5 and 1.7);
\draw [lightgray, fill] (3,-1) ellipse (1.5 and 1.7);
\draw [fill] (-3,-1) circle [radius = 0.07] node [black, above] {$x_0$};
\draw [fill] (3,-1) circle [radius = 0.07] node [black, above] {$y_0$};
\draw [fill] (-3,-2) circle [radius = 0.07] node [black, below] {$b$};
\draw [fill] (3,-2) circle [radius = 0.07] node [black, below] {$a$};
\draw (-3,-1) to (3,-1);
\draw (-3,-1) to (3,-2);
\draw (3,-1) to (-3,-2);
\draw [dashed] (3,-1) to (3, -2);
\draw [dashed] (-3,-1) to (-3, -2);
\draw [dashed] (-3,-2) to (3, -2);
\node at (-3,0.3) {$N_i(y_0)$};
\node at (3,0.3) {$N_j(x_0)$};
\end{tikzpicture}
\caption{Proof sketch for Lemma \ref{mainlmm}. The solid lines are edges that are known to be present. The dashed lines are edges that are likely to be present, due to small average hyperbolicity.\label{nnprime}}
\end{center}
\end{figure}
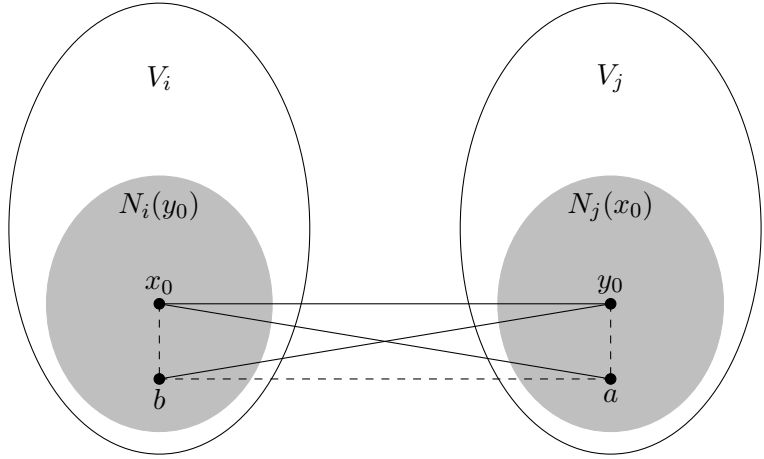

\begin{proof}[Proof of Lemma \ref{mainlmm}]
Throughout this proof, $C(\ep, m)$ denotes any positive real number that depends only on $\ep$ and $m$. The value of $C(\ep, m)$ may change from line to line.  For $x \in S'$, let $N(x)$ denote the neighborhood of $x$ in $G$.  Let $N_k(x) := N(x) \cap V_k$ for each $k$.  Let $V_i$ and $V_j$ be as in the statement of the lemma. Since $d(V_i,V_j) \geq 3\ep$, we have $\rho(V_i,V_j) \geq 3\ep\P(V_i)\P(V_j)$, and so there is some $x_0 \in V_i$ for which 
\begin{equation}\label{njvj}
 \P(N_j(x_0)) \geq 3\ep \P(V_j). 
 \end{equation}
By $\ep$-regularity,
\[ d(V_i,N_j(x_0)) \ge  d(V_i,V_j) - \ep \ge 2\ep, \]
and therefore
\begin{equation}\label{rhovivj}
 \rho( V_i, N_j(x_0) ) \ge 2\ep \P(V_i)\P(N_j(x_0)).
 \end{equation}
Now notice that
\begin{align*}
\rho( V_i, N_j(x_0) ) &= \rho(V_i, N_j(x_0) \cap B(x_0)) + \rho(V_i, N_j(x_0) \cap B(x_0)^c) \\
&\leq \rho(V_i, B(x_0)) + \rho( V_i, N_j(x_0) \cap B(x_0)^c ) \\
&\leq \P(V_i) \P(B(x_0)) + \rho( V_i, N_j(x_0) \cap B(x_0)^c ).
\end{align*}
Since $x_0 \notin A$, $\P(B(x_0)) \leq \delta_0$.  Thus
\begin{align*}
\rho(V_i,N_j(x_0)) &\leq \left( { \delta_0 \over \P(N_j(x_0))} \right) \P(V_i) \P(N_j(x_0)) \notag\\
&\qquad + \rho( V_i, N_j(x_0) \cap B(x_0)^c ),
\end{align*}
so that by \eqref{rhovivj},
\begin{align}
\biggl(2\ep  -{ \delta_0 \over \P(N_j(x_0)) }  \biggr) \P(V_i) \P(N_j(x_0)) \le \rho(V_i, N_j(x_0) \cap B(x_0)^c ).\label{rhoninj}
\end{align}
By Lemma~\ref{vlmm} and the inequality \eqref{njvj}, 
\[ \P(N_j(x_0)) \geq 3\ep \P(V_j) \geq  C(\ep,m)\P(S'). \]
Combining this with \eqref{rhoninj}, we get
\begin{align*}
\biggl( 2\ep  - { \delta_0 \over C(\ep,m)\P(S')} \biggr) \P(V_i) \P(N_j(x_0)) \le \rho(V_i, N_j(x_0) \cap B(x_0)^c ).
\end{align*}
If $\delta_0$ is sufficiently small (depending on $\ep$, $m$ and $\P(S')$), the quantity in brackets on the left is bounded below by $\ep$, and 
so there is $y_0 \in N_j(x_0) \cap B(x_0)^c$ such that
\begin{equation}\label{nivi}
 \P(N_i(y_0)) \geq \ep \P(V_i).
 \end{equation}
Recalling \eqref{njvj}, we see that by $\ep$-regularity,
\begin{equation}\label{dvivj}
  d(V_i,V_j) \ge d(N_j(x_0),N_i(y_0)) - \ep. 
  \end{equation}
The quantity $d(N_j(x_0),N_i(y_0))$ can be bounded from below as follows:
\begin{align*}
d(N_j(x_0),N_i(y_0)) &= { \P^{\otimes 2}((a,b) \in N_j(x_0) \times N_i(y_0) \colon s(a,b) \geq t ) \over \P(N_j(x_0)) \P(N_i(y_0)) } \\
&\geq { \P^{\otimes 2}( (a,b) \in N_j(x_0) \times N_i(y_0) \colon s(a,b),s(a,y_0) \geq t ) \over \P(N_j(x_0)) \P(N_i(y_0)) }. 
\end{align*}
We wish to show that the right side is close to $1$. For that purpose, we write the right side as $(1-(i))(1-(ii))$, where 
\begin{align*}
(i) &\coloneqq 1 - { \P^{\otimes 2}( (a,b) \in N_j(x_0) \times N_i(y_0) \colon s(a,b), s(a,y_0) \geq t ) \over \P( a \in N_j(x_0) \colon s(a,y_0) \geq t ) \P(N_i(y_0))} \\
&= { \P^{\otimes 2}( (a,b) \in N_j(x_0) \times N_i(y_0) \colon s(a,b) < t \leq s(a,y_0) ) \over \P( a \in N_j(x_0) \colon s(a,y_0) \geq t ) \P(N_i(y_0))}
\end{align*}
and
\begin{align*}
(ii) &\coloneqq 1 - { \P( a \in N_j(x_0) \colon s(a,y_0) \geq t ) \over \P( N_j(x_0) )} \\
&= { \P(a \in N_j(x_0) \colon s(a,y_0) < t ) \over  \P(N_j(x_0)) }.
\end{align*}
We will now show that $(i)$ and $(ii)$ are small. (To understand heuristically why they should be small, recall Figure \ref{nnprime}.) Recalling the definition of $\frak{R}^2(y_0)$, we see that
\begin{align*}
\frak{R}^2(y_0) &\supset \{ (a,b) \in N_j(x_0) \times N_i(y_0) \colon s(a,b) < t \leq \min\{ s(a,y_0),s(b,y_0) \} \}.
\end{align*}
But if $b \in N_i(y_0)$, then $b$ is a neighbor of $y_0$ in $G$ and so $s(b,y_0) \geq t$.  Thus the above display can be simplified to 
\begin{align*}
 \frak{R}^2(y_0) \supset \{ (a,b) \in N_j(x_0) \times N_i(y_0) \colon s(a,b) < t \leq s(a,y_0) \}.
 \end{align*}
Moreover, recalling that $y_0 \in N_j(x_0)$, so that $s(x_0,y_0) \geq t$, and recalling the definition of $\frak{R}^1(y,z)$, 
it is easy to see that
\begin{align}
&\P(a \in N_j(x_0) \colon s(a,y_0) < t) \notag\\
&\leq \P(a:s(a,y_0) < t \leq \min\{ s(a,x_0), s(x_0,y_0) \} )\notag\\
&\leq \P(\frak{R}^1(y_0,x_0)).\label{ibd1}
\end{align}
Thus,
\begin{align*}
\P( a \in N_j(x_0) \colon s(a,y_0) \geq t ) &\geq \P(N_j(x_0)) - \P(\frak{R}^1(y_0,x_0)).
\end{align*}
By \eqref{njvj} and \eqref{nivi}, $\P(N_j(x_0))$ and $\P(N_i(y_0))$ are both bounded below by $C(\ep,m)\P(S')$. Since $y_0\notin A$, Lemma \ref{r2small} gives
\[
\P^{\otimes 2}(\frak{R}^2(y_0)) \leq 2\delta_0.
\]
On the other hand, since $y_0\notin B(x_0)$, 
\begin{align*}
\P(\frak{R}^1(y_0, x_0) ) \leq \delta_0.
\end{align*}
Combining all of the above observations, we get
\begin{align*}
(i) &\leq { \P^{\otimes 2}(\frak{R}^2(y_0)) \over \left( \P(N_j(x_0)) - \P(\frak{R}^1(y_0,x_0))\right)\P(N_i(y_0))} \\
&\leq {2\delta_0 \over (C(\ep, m)\P(S')-\delta_0)C(\ep,m)\P(S')}.
\end{align*}
If $\delta_0$ is small enough (depending on $\ep$, $m$ and $\P(S')$), the above quantity is smaller than $\ep/2$. 
For $(ii)$, we re-use \eqref{ibd1} to get 
\begin{align*}
(ii) &\leq { \P(\frak{R}^1(y_0,x_0)) \over  \P(N_j(x_0)) }  \leq \frac{\delta_0}{C(\ep,m)\P(S')}.
\end{align*}
Again, this is smaller than $\ep/2$ if $\delta_0$ is small enough. Thus, 
\[ d(N_j(x_0),N_i(y_0)) \geq 1 - (i) - (ii)\geq 1-\ep, \]
and hence by \eqref{dvivj}, $d(V_i,V_j)\ge 1-2\ep$.
\end{proof}
Our second key lemma shows that the property of high density between regular pairs has a certain transitivity property. 
\begin{lmm}\label{chainlmm}
There exists a number $\delta^*$ depending only on $\ep$, $m$ and $\P(S')$, such that if $\delta_0\le \delta^*$, then the following holds. Suppose that $(V_a, V_b)$ is an $\epsilon$-regular pair. Suppose that $i_0, i_1,\ldots, i_k$ are distinct elements of $\{1,\ldots,q\}$ such that $i_0=a$, $i_k=b$, $d(V_{i_j}, V_{i_{j+1}})\ge 1-2\ep$ for each $0\le j\le k-1$, and $2\le k\le \ep^{-1/2}$. Then $d(V_a, V_b)\ge 1-2\ep$. 
\end{lmm}
The proof of this lemma is intuitively quite simple, given that we already have Lemma \ref{mainlmm}. The small hyperbolicity ensures that if we have a path in $G$ that is not too long, then it is likely that the beginning and ending points of the path are connected by an edge. This allows us to conclude that $d(V_a, V_b)$ is close to $1$, as long as $k$ is not too large. In particular, $d(V_a, V_b)\ge 3\ep$. But then Lemma \ref{mainlmm} implies that $d(V_a, V_b)\ge 1-2\ep$. 
\begin{proof}[Proof of Lemma \ref{chainlmm}]
Take any sequence of points $x_j\in V_{i_j}$, $0\le j\le k$, such  that for each $0\le j\le k-1$, $s(x_j, x_{j+1})\ge t$, and $s(x_0,x_k)<t$. Let $L$ be the set of all such sequences ($L$ is allowed to be empty). Since $s(x_0,x_k)<t$, then there is a minimum $j$ such that $s(x_0,x_j)<t$. But $s(x_0,x_1)\ge t$. Thus, $j\ge 2$, and hence $s(x_0,x_{j-1})\ge t$. But we also know that $s(x_{j-1},x_j)\ge t$. Therefore, $(x_0,x_j, x_{j-1})\in R_t$, where $R_t$ is the set defined in \eqref{rtdef}. Since $t\notin R$ and $k\le \ep^{-1/2}$, this implies that 
\begin{align}
\sum_{(x_0,\ldots,x_k)\in L} \P(x_0)\cdots \P(x_k) &\le \sum_{j=2}^k \sum_{\substack{x_0,\ldots,x_k\in S,\\ (x_0,x_j,x_{j-1})\in R_t}}\P(x_0)\cdots \P(x_k)\notag \\
&\le k \P(R_t) \le \frac{\delta_0^4}{\sqrt{\ep}}.\label{pxk1}
\end{align}
On the other hand, let $B := V_{i_0}\times \cdots \times V_{i_k}$. Then 
\begin{align*}
&\sum_{(x_0,\ldots,x_k)\in B\setminus L} \P(x_0)\cdots \P(x_k) \\
&\le \sum_{j=0}^{k-1} \sum_{\substack{(x_0,\ldots,x_k)\in B\\s(x_j,x_{j+1}) < t}}  \P(x_0)\cdots \P(x_k)+ \sum_{\substack{(x_0,\ldots,x_k)\in B\\ s(x_0,x_k) \ge t}}  \P(x_0)\cdots \P(x_k)\\
&= \P(V_{i_0})\cdots \P(V_{i_k}) \biggl( \sum_{j=0}^{k-1} (1-d(V_{i_j}, V_{i_{j+1}})) + d(V_a, V_b)\biggr)\\
&\le \P(V_{i_0})\cdots \P(V_{i_k}) (2k\ep + d(V_a, V_b)) \\
&\le  \P(V_{i_0})\cdots \P(V_{i_k}) (2\sqrt{\ep} + d(V_a, V_b)).
\end{align*}
But by \eqref{pxk1}, 
\begin{align*}
&\sum_{(x_0,\ldots,x_k)\in B\setminus L} \P(x_0)\cdots \P(x_k) \\
&= \sum_{(x_0,\ldots,x_k)\in B} \P(x_0)\cdots \P(x_k) - \sum_{(x_0,\ldots,x_k)\in L} \P(x_0)\cdots \P(x_k)\\
&\ge \P(V_{i_0})\cdots \P(V_{i_k}) - \frac{\delta_0^4}{\sqrt{\ep}}.
\end{align*}
Combining the last two displays, we get
\begin{align*}
d(V_a, V_b) &\ge 1 - \frac{\delta_0^4}{\sqrt{\ep} \P(V_{i_0})\cdots \P(V_{i_k})} -2\sqrt{\ep}. 
\end{align*}
By Lemma \ref{vlmm}, this shows that if $\delta_0$ is sufficiently small (depending on $\ep$, $m$ and $\P(S')$), then 
\[
d(V_a, V_b) \ge 1-3\sqrt{\ep}.
\]
But then by Lemma \ref{mainlmm} (assuming that $\ep$ is sufficiently small), this gives  $d(V_a, V_b)\ge 1-2\ep$. 
\end{proof}

We now begin the main quest of this section, namely, to show that a small fraction of the edges of $G$ can be modified to transform it into a disjoint union of cliques. Throughout the rest of this section, we will assume that: 
\begin{align}
&\textit{$\delta_0$ is so small, depending on $\ep$, $m$ and $\P(S')$, that the}\notag\\
&\textit{conclusions of Lemma \ref{mainlmm} and Lemma \ref{chainlmm} hold.}\label{delta0}
\end{align}
First, we define a graph structure on $\{V_1,\ldots,V_q\}$. We will say that there is an edge between $V_i$ and $V_j$ if $(V_i, V_j)$ is $\ep$-regular and $d(V_i, V_j)\ge 1-2\ep$. In this case we will say that $V_i$ and $V_j$ are neighbors. A subset $\mn$ of $\{V_1,\ldots, V_q\}$ will be called a ``neighborhood'' if there is some $V_i\in \mn$ such that all other  elements of $\mn$ are neighbors of $V_i$.  In this case we will say that $\mn$ is a neighborhood of $V_i$. Note that $\mn$ need not contain all the neighbors of $V_i$. Let $\mfn$ be a maximal collection of disjoint neighborhoods such that each neighborhood has size $\ge \ep^{1/4}q$. Note that $\mfn$ is allowed to be empty, in case there is no neighborhood of size $\ge \ep^{1/4}q$. 
\begin{lmm}\label{regexist}
For any distinct $\mn_1,\mn_2\in \mfn$, there is some $V_i\in \mn_1$ and $V_j\in \mn_2$ such that $(V_i, V_j)$ is an $\ep$-regular pair.
\end{lmm}
\begin{proof}
Since $|\mn_1|$ and $|\mn_2|$ are both $\ge \ep^{1/4}q$, there are at least $\ep^{1/2} q^2$ pairs $(V_a,V_b)$ such that $V_a\in \mn_1$ and $V_b\in \mn_2$. Since the number of irregular pairs is at most $\ep q^2$, this shows that at least one of the above pairs must be $\ep$-regular. 
\end{proof}
Now define a graph structure on $\mfn$ as follows. Say that two neighborhoods $\mn_1,\mn_2\in \mfn$ are connected by an edge if there exists $V_i\in \mn_1$ and $V_j\in \mn_2$ such that $V_i$ and $V_j$ are neighbors (in the sense defined above). 
\begin{lmm}\label{firstclique}
Under the graph structure defined above, $\mfn$ is a disjoint union of cliques.
\end{lmm}
\begin{proof}
For distinct $\mn_1,\mn_2,\mn_3\in \mfn$, we have to show that if $\mn_1$ is a neighbor of $\mn_2$, and $\mn_3$ is a neighbor of $\mn_2$, then $\mn_3$ is a neighbor of $\mn_1$. This will imply that $\mfn$ is a disjoint union of cliques. 

Accordingly, let $V_i\in \mn_1$ and $V_j\in \mn_2$ be neighbors, and let $V_k\in \mn_2$ and $V_l\in \mn_3$ be neighbors.  By Lemma \ref{regexist}, there is an $\ep$-regular pair $(V_a, V_b)$ such that $V_a\in \mn_1$ and $V_b\in \mn_3$. Suppose that $\mn_i$ is a neighborhood of $V_{t_i}$, for $i=1,2,3$. Then the sequence $V_a, V_{t_1}, V_i, V_j, V_{t_2}, V_k, V_l, V_{t_3}, V_b$ is a path in the graph defined on $\{V_1,\ldots,V_q\}$ (see Figure~\ref{fcfig}). Since $(V_a, V_b)$ is $\ep$-regular, Lemma \ref{chainlmm} implies that $d(V_a, V_b)\ge 1-2\ep$. In other words, $V_a$ and $V_b$ are neighbors. Thus, $\mn_1$ is a neighbor of $\mn_3$. 
\end{proof}
\begin{figure}
\begin{center}
\begin{tikzpicture}
\draw (-3,0) ellipse (1.3 and 3);
\draw (0,0) ellipse (1.3 and 3);
\draw (3,0) ellipse (1.3 and 3);
\draw [fill] (-3, 0) circle [radius = 0.07] node [black,right] {$V_{t_1}$};
\draw [fill] (0, 0) circle [radius = 0.07] node [black,below] {$V_{t_2}$};
\draw [fill] (3, 0) circle [radius = 0.07] node [black,right] {$V_{t_3}$};
\draw [fill] (-3, 2) circle [radius = 0.07] node [black,above] {$V_i$};
\draw [fill] (0, 2) circle [radius = 0.07] node [black,above] {$V_j$};
\draw [fill] (3, 1) circle [radius = 0.07] node [black,above] {$V_l$};
\draw [fill] (.7, 1) circle [radius = 0.07] node [black,above] {$V_k$};
\draw [fill] (-2.5, -2) circle [radius = 0.07] node [black,below] {$V_a$};
\draw [fill] (2.5, -2) circle [radius = 0.07] node [black,below] {$V_b$};
\draw (-2.5,-2) to (-3,0) to (-3,2) to (0,2) to (0,0) 
to (.7,1) to (3,1) to (3,0) to (2.5, -2);
\draw [dashed] (-2.5,-2) to (2.5, -2);
\node at (-3, -3.5) {$\mn_1$};
\node at (0, -3.5) {$\mn_2$};
\node at (3, -3.5) {$\mn_3$};
\end{tikzpicture}
\caption{Illustration of the proof of Lemma \ref{firstclique}. The solid lines are known to be edges in the graph defined on $\{V_1,\ldots, V_q\}$. We deduce that the dashed line is also an edge, by invoking Lemma \ref{chainlmm}. \label{fcfig}}
\end{center}
\end{figure}
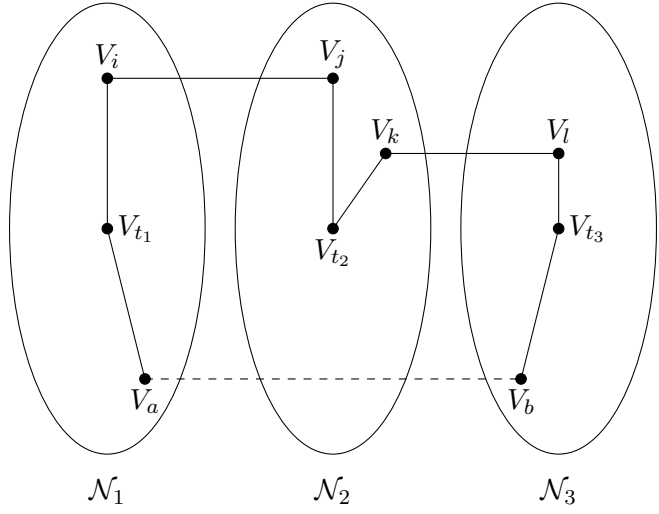
Take each clique in $\mfn$, and take the union of its elements. This yields a new collection $\mfc$ of disjoint subsets of $\{V_1,\ldots, V_q\}$. 
\begin{lmm}\label{csize}
We have $|\mfc|\le \ep^{-1/4}$. 
\end{lmm}
\begin{proof}
Simply note that each $\mc\in \mfc$ has size at least $\ep^{1/4}q$, these sets are disjoint, and their union is a subset of $\{V_1,\ldots, V_q\}$. Thus, $|\mfc|\ep^{1/4}q\le q$.
\end{proof}
\begin{lmm}\label{c1c2}
If $V_i\in \mc_1$ and $V_j\in \mc_2$ for two distinct elements $\mc_1$ and $\mc_2$ of $\mfc$, then $V_i$ and $V_j$ are not neighbors. On the other hand, if $V_i, V_j\in \mc$ for some $\mc\in \mfc$, then either $(V_i,V_j)$ is an irregular pair, or $V_i$ and $V_j$ are neighbors. Moreover, in this case even if $(V_i,V_j)$ is irregular, there is a path with $\le 6$ vertices joining $V_i$ and $V_j$. 
\end{lmm}
\begin{proof}
If $V_i\in \mc_1$ and $V_j\in \mc_2$ for two distinct elements $\mc_1$ and $\mc_2$ of $\mfc$, it follows directly from the definition of $\mfc$ that $V_i$ and $V_j$ cannot be neighbors. Next, suppose that $V_i, V_j\in \mc$ for some $\mc\in \mfc$, and $(V_i, V_j)$ is $\ep$-regular. Then either $V_i,V_j\in \mn$ for some $\mn\in \mfn$, or  $V_i\in \mn_1$ and $V_j\in \mn_2$ for some $\mn_1,\mn_2\in \mfn$ that are neighbors. In the first case, suppose that $\mn$ is a neighborhood of some $V_a$. Then $V_i, V_a, V_j$ is a path, and hence by Lemma \ref{chainlmm}, $V_i$ is a neighbor of $V_j$. In the second case, suppose that $\mn_1$ is a  neighborhood of $V_a$ and $\mn_2$ is a neighborhood of $V_b$. Since $\mn_1$ and $\mn_2$ are neighbors, there exist $V_k\in \mn_1$ and $V_l\in \mn_2$ which are neighbors. Then $V_i, V_a, V_k, V_l, V_b, V_j$ is a path, and hence by Lemma \ref{chainlmm}, $V_i$ and $V_j$ are neighbors. This argument also establishes that even if $(V_i, V_j)$ is an irregular pair, we can find a path with $\le 6$ vertices joining $V_i$ and $V_j$.
\end{proof}
Next, let $\md$ be the set of all $V_i$ that are not elements of any $\mc\in \mfc$. 
\begin{lmm}\label{outsidenbrs}
For any $V_i\in \md$, there are less than $\ep^{1/4} q$ many $V_j\in \md$ that are neighbors of $V_i$. 
\end{lmm}
\begin{proof}
Suppose that there is some $V_i\in \md$ that has $\ge \ep^{1/4} q$ neighbors in $\md$. Then there is a neighborhood $\mn\subset \md$ of size $\ge \ep^{1/4} q$. But this neighborhood is disjoint from all the neighborhoods in $\mfn$. This contradicts the maximality of $\mfn$. 
\end{proof}
\begin{lmm}\label{uniquegroup}
Suppose that $V_i \in \md$ and $\mc \in \mfc$ are such that $V_i$ has at least $\ep^{1/3}q$ neighbors in $\mc$. Then $V_i$ has less than $\ep^{1/3} q$ neighbors in the union of all members of $\mfc$ other than $\mc$. 
\end{lmm}
\begin{proof}
Let $\ms_1$ be the set of all neighbors of $V_i$ in $\mc$, and let $\ms_2$ be the set of all neighbors of $V_i$ in the union of all elements of $\mfc$ other than $\mc$. By assumption, $|\ms_1|\ge \ep^{1/3}q$. If also $|\ms_2|\ge \ep^{1/3} q$, then there are $\ge \ep^{2/3}q^2$ pairs $(V_j, V_k)$ such that $V_j \in \ms_1$ and $V_k\in \ms_2$. Therefore at least one such pair $(V_j, V_k)$ must be $\ep$-regular. Since $V_j, V_i, V_k$ is a path, Lemma \ref{chainlmm} shows that $V_j$ and $V_k$ are neighbors. But this contradicts the first assertion of Lemma~\ref{c1c2}. 
\end{proof}
For each $\mc\in \mfc$, let $\mc'$ be the superset of $\mc$ consisting of all elements of $\mc$ and all elements of $\md$ that have $\ge \ep^{1/3}q$ neighbors in $\mc$. Let $\mfc'$ be the set of all such $\mc'$.
Lemma \ref{uniquegroup} shows for any $V_i\in \md$, there can be at most one $\mc\in \mfc$ such that $V_i$ has $\ge \ep^{1/3}q$ neighbors in $\mc$.  Thus, the elements of $\mfc'$ are disjoint. Let $\md'$ be the set of all elements of $\md$ that do not belong to any $\mc'$. 
 A schematic picture depicting $\mfc'$ and $\md'$ is given in Figure \ref{mfcpic}.
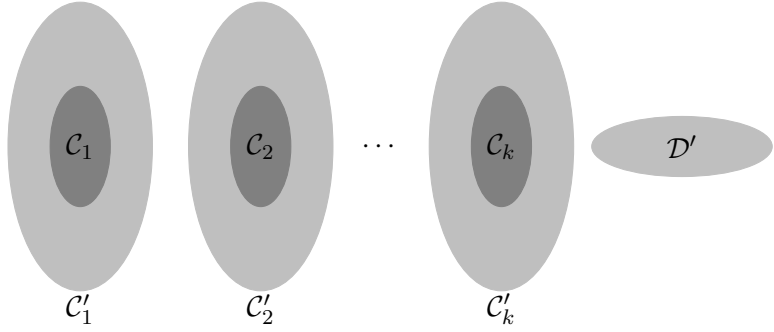
\begin{figure}
\begin{center}
\begin{tikzpicture}[scale=.8]
\draw [lightgray, fill] (0,0) ellipse (1.2 and 2.4);
\draw [lightgray, fill] (-4,0) ellipse (1.2 and 2.4);
\draw [lightgray, fill] (-7,0) ellipse (1.2 and 2.4);
\draw [lightgray, fill] (3,0) ellipse (1.5 and .5);
\draw [gray, fill] (0,0) ellipse (.5 and 1);
\draw [gray, fill] (-4,0) ellipse (.5 and 1);
\draw [gray, fill] (-7,0) ellipse (.5 and 1);
\node at (-2,0) {$\cdots$};
\node at (0,0) {$\mc_k$};
\node at (-4,0) {$\mc_2$};
\node at (-7,0) {$\mc_1$};
\node at (3,0) {$\md'$};
\node at (0,-2.7) {$\mc_k'$};
\node at (-4,-2.7) {$\mc_2'$};
\node at (-7,-2.7) {$\mc_1'$};
\end{tikzpicture}
\caption{Schematic picture of the components of $\mfc'$ (where $k$ is the number of components) and the remainder set $\md'$. The union of the light gray regions is $\md$.\label{mfcpic}}
\end{center}
\end{figure}

\begin{lmm}\label{internalpairs}
For any $\mc\in \mfc$, the set $\mc'$ has the property that any two distinct elements of $\mc'$ are either neighbors, or an irregular pair.
\end{lmm}
\begin{proof}
Take any distinct $V_i, V_j\in \mc'$ such that $(V_i,V_j)$ is an $\ep$-regular pair. If they are both in $\mc$, then the assertion is proved by Lemma \ref{c1c2}. 

If $V_i\in \mc$ and $V_j \in \md$, then $V_j$ has a neighbor $V_k \in \mc$. By Lemma \ref{c1c2}, there is a path with $\le 6$ vertices joining $V_k$ and $V_i$. Since $V_j$ and $V_k$ are neighbors, we can concatenate $V_j$ at the beginning of this path to get a path with $\le 7$ vertices joining $V_j$ and $V_i$. Therefore by Lemma \ref{chainlmm}, $V_j$ and $V_i$ are neighbors.

Lastly, if $V_i$ and $V_j$ are both in $\md$, then they have neighbors $V_k$ and $V_l$ in $\mc$. By Lemma \ref{c1c2}, there is a path with $\le 6$ vertices joining $V_k$ and $V_l$. Since $V_i$ and $V_k$ are neighbors, and $V_j$ and $V_l$ are neighbors, we can concatenate $V_i$ at the beginning of the path and $V_j$ to the end of the path to get a path with $\le 8$ vertices joining $V_i$ and $V_j$. Therefore by Lemma \ref{chainlmm}, $V_i$ and $V_j$ are neighbors.
\end{proof}
Call a pair $(V_i, V_j)$ ``bad'' if $V_i$ and $V_j$ are neighbors, but they belong to distinct elements of $\mfc'$. 
\begin{lmm}\label{badpairs}
The number of bad pairs is at most $3\ep^{1/12} q^2$. 
\end{lmm}
\begin{proof}
Let $(V_i,V_j)$ be a bad pair. We consider several cases. First, by Lemma \ref{c1c2}, it cannot be that both $V_i$ and $V_j$ are in the complement of $\md$. 

Next, suppose that $V_i\in \md$ and $V_j\notin \md$. Then $V_i\in \mc_1'$ for some $\mc_1\in \mfc$ and $V_j\in \mc_2$ for some $\mc_2\ne \mc_1$. By Lemma \ref{uniquegroup}, there are less than $\ep^{1/3}q$ neighbors of $V_i$ in $\mc_2$. By Lemma \ref{csize}, there are at most $\ep^{-1/4}$ choices of $\mc_2$. Thus, there are at most $\ep^{-1/4}\ep^{1/3}q= \ep^{1/12} q$ choices of $V_j$ for this $V_i$, and therefore at most  $\ep^{1/12} q^2$ choices of $(V_i,V_j)$ of this type.

Finally, suppose that both $V_i, V_j\in \md$. Then by Lemma \ref{outsidenbrs}, there are less than $\ep^{1/4}q$ choices of $V_j$ for each $V_i$. Thus, there are at most $\ep^{1/4}q^2$ pairs of this type. 
\end{proof}

\begin{lmm}\label{dlmm}
Any element of $\md'$ has at most $2\ep^{1/12}q$ neighbors among $\{V_1,\ldots, V_q\}$.
\end{lmm}
\begin{proof}
Take any $V_i\in \md'$ and any neighbor $V_j$ of $V_i$. Then by Lemma \ref{outsidenbrs}, there are less than $\ep^{1/4}q$ choices of $V_j\in \md$. On the other hand, by definition of $\md'$, $V_i$ has less than $\ep^{1/3}q$ neighbors in each $\mc\in \mfc$. Thus, by Lemma~\ref{csize}, there are at most $\ep^{1/12}q$ choices of such $V_j$. Since any neighbor of $V_i$ is either in $\md$ or in $\mc$ for some $\mc\in \mfc$, this completes the proof.  
\end{proof}

We finally arrive at the main result of this section, which says that the graph $G$ can be modified into a disjoint union of cliques by adding and deleting a set of edges that has small $\P^{\otimes 2}$-measure.

\begin{lmm}\label{cliquelmm}
Under the assumptions \eqref{pps} and \eqref{delta0}, the graph $G$ can be modified into a disjoint union of cliques by adding and deleting edges in such a way that if $\Delta E$ is the set of all edges that were added or deleted, then 
\begin{equation}\label{diffbd}
\P^{\otimes 2}(\Delta E) \le C(\ep^{1/12} + m^{-1})\P(S')^2,
\end{equation}
where $C$ is a universal constant. 
Moreover, any non-singleton clique $B$ in the resulting graph  has 
\begin{align}\label{pblower}
\P(B)\ge \frac{1}{2}\ep^{1/4} \P(S').
\end{align}
\end{lmm}
\begin{proof}
Edges are added and deleted in several steps. First, delete all edges with at least one endpoint in $V_0$. Let $\Delta E_1$ be the set of deleted edges. Then clearly 
\[
\P^{\otimes 2}(\Delta E_1)\le 2\P(V_0) \P(S')\le 2\ep \P(S')^2. 
\]
Next, add all edges between vertices within the same $V_i$, $1\le i\le q$. Let $\Delta E_2$ be the set of all edges added in this step. Then by Lemma \ref{vlmm}, 
\begin{align*}
\P^{\otimes 2}(\Delta E_2) &\le \sum_{i=1}^q \P(V_i)^2 \le q \frac{9\P(S')^2}{4q^2}\\
&= \frac{9\P(S')^2}{4q} \le \frac{9\P(S')^2}{4m}.
\end{align*}
In the next step, add all missing edges between any $V_i$ and $V_j$ that are members of the same $\mc'\in \mfc'$. By Lemma \ref{internalpairs}, such pairs are either irregular, or they are neighbors of each other. In the latter case, the total mass of the missing edges is at most $2\ep \P(V_i)\P(V_j)$. Thus, if $\Delta E_3$ is the set of edges added in this step, then by Lemma \ref{vlmm},
\[
\P^{\otimes 2}(\Delta E_3) \le (\ep q^2  + 2\ep q^2) \frac{9\P(S')^2}{4q^2}\le 7\ep\P(S')^2. 
\]
Next, delete all edges between any $V_i\in \mc_1'$ and $V_j\in \mc_2'$ where $\mc_1'\ne\mc_2'$. Then $(V_i, V_j)$ is either an irregular pair, or $(V_i,V_j)$ is regular but $V_i$ and $V_j$ are not neighbors, or $(V_i,V_j)$ is a bad pair. Thus, if $\Delta E_4$ is the set of edges added in this step, then by Lemma \ref{mainlmm}, Lemma \ref{badpairs} and Lemma \ref{vlmm},
\begin{align*}
\P^{\otimes 2}(\Delta E_4) &\le (\ep q^2  + 3\ep q^2 + 3\ep^{1/12} q^2) \frac{9\P(S')^2}{4q^2} \\
&\le 16\ep^{1/12}\P(S')^2.
\end{align*}
Finally, delete all edges with at least one vertex in some $V_i\in \md'$. Let $\Delta E_5$ be the set of deleted edges. Given $V_i\in\md'$ and any $V_j$, by Lemma \ref{dlmm} there are at most $2\ep^{1/12} q$ choices of $V_j$ such that $V_j$ is a neighbor of $V_i$. The other possibilities are that $(V_i, V_j)$ is an irregular pair, or $(V_i,V_j)$ is regular but $V_j$ is not a neighbor of $V_i$, or $V_j=V_i$. Therefore by Lemma \ref{mainlmm} and Lemma~\ref{vlmm},
\begin{align*}
\P^{\otimes 2}(\Delta E_5) &\le (\ep q^2  + 3\ep q^2 + 2\ep^{1/12} q^2 + q)\frac{9\P(S')^2}{4q^2}\\
&\le (14\ep^{1/12}+ 3m^{-1})\P(S')^2.
\end{align*}
This completes the process of adding and deleting edges. If $\Delta E$ is the set of all edges that were either added or deleted, then the above estimates show that \eqref{diffbd} holds. 

Let us now verify that the resulting graph is a disjoint union of cliques. For each $\mc'\in \mfc'$, let $V(\mc')$ be the union of all $V\in \mc'$. In the new graph, each $V(\mc')$ is a clique, and there are no edges between two such cliques. Moreover, any vertex that belongs to some $V_i\in \md'$ has no edges incident to it in the new graph. Thus, the new graph is the disjoint union of the above cliques and a bunch of singleton vertices that are disconnected from all else. This also shows that any non-singleton clique in the new graph must be one of the $V(\mc')$'s. But for any $\mc'\in \mfc'$, Lemma \ref{vlmm} gives
\begin{align*}
\P(V(\mc')) &= \sum_{V\in \mc'} \P(V) \\
&\ge |\mc'|\frac{\P(S')}{2q} \ge |\mc|\frac{\P(S')}{2q}\ge \ep^{1/4}q\frac{\P(S')}{2q} = \frac{1}{2}\ep^{1/4}\P(S').
\end{align*}
This completes the proof. 
\end{proof}

\section{Constructing the tree}\label{treesec}
Let $P^*$, $\delta_0$, $A$, $\ep$, $m$, $\kappa$, $N$ and $t_1,\ldots, t_N$ remain as defined in Section~\ref{setup}.   We will now repeatedly apply Lemma \ref{cliquelmm} to extract from $S$ a nested hierarchy of subsets with desirable properties. The subsets will be constructed in such a way that each subset is either a singleton, or has $\P$-measure uniformly bounded below by a positive constant that depends only on $\ep$ and $m$. Any such constant will henceforth be denoted by $C(\ep, m)$. This will allow us to apply Lemma \ref{cliquelmm} to partition such a non-singleton subset if $P^*$ and $\delta_0$ are small enough, depending only on $\ep$ and $m$. We will keep dividing the non-singleton subsets until we are left with only singletons. 

Henceforth, whenever we say ``$\delta_0$ and $P^*$ are small enough'', we will mean ``$\delta_0$ and $P^*$ are smaller than constants depending only on $\ep$ and $m$''.

Let $S' =S\setminus A$. By Lemma \ref{asmall}, $\P(S')\ge 1/2$ if $\delta_0$ is small enough. Define a graph on $S'$ as in the beginning of Section \ref{modifyingG}, using $t=t_1$, and obtain a partition of $S'$ using Lemma~\ref{cliquelmm}. Obtain a partition of $S$ by taking this partition of $S'$ and appending to it singleton sets consisting of the elements of $A$. Let $\mv_1$ denote this partition. By \eqref{pblower}, any non-singleton element $V\in\mv_1$ does not intersect $A$ and satisfies $\P(V)\ge C(\ep,m)$. Thus we can apply Lemma \ref{cliquelmm} to any such $V$ with $t=t_2$, if $\delta_0$ and $P^*$ are small enough. In this manner, we obtain a collection $\mv_2$ of disjoint sets, each of which is a subset of some non-singleton element of $\mv_1$. Then we partition each non-singleton element of $\mv_2$ by applying the procedure of Section \ref{modifyingG} with $t=t_3$ to obtain $\mv_3$, and continue this recursive partitioning until we arrive at $\mv_N$. This is possible since $N\le C(\ep,m)$, which, by \eqref{pblower}, ensures that the conditions \eqref{pps} and \eqref{delta0}  are  never violated if $\delta_0$ and $P^*$ are small enough.

Having defined $\mv_1,\ldots, \mv_N$, define $\mv_{N+1}$ to be the set of all singleton sets $\{x\}$ such that $x$ belongs to some non-singleton member of $\mv_N$. Note, in particular, that we are not applying Lemma \ref{cliquelmm} while partitioning the elements of $\mv_{N}$ into singletons.  Lastly, define $\mv_0:=\{S\}$. 

Let $T$ be the set of all pairs $(i,V)$ where $0\le i\le N+1$ and $V\in \mv_i$. This is sort of like the union of the $\mv_i$'s, except that we pair each element $V$ with the corresponding $i$ to deal with the problem of the same $V$ appearing in two different $\mv_i$'s (which can happen if some $V$ is partitioned into just one set in some step). For simplicity, we will refer to the element $(i,V)\in T$ as just $V$.

We will now define a tree structure on $T$. Note that by construction, if an element $V\in T$ belongs to some $\mv_i$, $i\ge 1$, then it has a uniquely defined parent $U\in \mv_{i-1}$. Putting edges between such parent-child pairs creates a graph which is obviously a tree.  Also, it is clear that the set of leaves of this tree can be identified with $S$. Define  $r := (0,S)$ 
to be the root of $T$. 

For each non-singleton node $V\in \mv_i$ for $1\le i\le N-1$, let $\Delta E(V)$ be the set of edges of $V$ that need to be modified while applying Lemma~\ref{cliquelmm} to convert $V$ into a disjoint union of cliques. If $V$ is a singleton set, let $\Delta E(V)$ be empty. Let $\Delta E(S')$ be the set of edges that need to be modified while applying Lemma~\ref{cliquelmm} to $S'$. Lastly, let $\Delta E(A)$ be the set of all pairs $(x,y)$ with  at least one of $x$ and $y$ in~$A$. Let $\Delta E$ be the union of all these sets. 

We prove three lemmas in this section. In all of these, we assume that $P^*$ and $\delta_0$ are sufficiently small, depending on $\ep$ and $m$, so that Lemma \ref{cliquelmm} can be applied. We will view the elements of $S$ as the leaves of $T$, and for any $x,y\in S$, we will denote by $(x,y)_r$ the Gromov product of $x$ and $y$ under the graph distance on $T$, with respect to the base point $r$. 
\begin{lmm}\label{endlmm1}
For the set $\Delta E$ defined above, we have
\[
\P^{\otimes 2} (\Delta E) \le C \ep^{1/24} + C m^{-1/2} + 2\delta_0,
\]
where $C$ is a universal constant. 
\end{lmm}
\begin{proof}
Note that by Lemma~\ref{cliquelmm} and Lemma~\ref{asmall},
\begin{align*}
\P^{\otimes 2} (\Delta E) &\le \P^{\otimes 2} (\Delta E(S')) +\sum_{i=1}^{N-1} \sum_{V\in \mv_i}\P^{\otimes 2} (\Delta E(V)) + 2\P(A)\\
&\le  C(\ep^{1/12} + m^{-1}) \biggl(\P(S')^2 + \sum_{i=1}^{N-1} \sum_{V\in \mv_i} \P(V)^2\biggr) + 2\delta_0.
\end{align*}
Since each $\mv_i$ is a partition of a subset of $S$, 
\[
\sum_{V\in \mv_i} \P(V)^2\le \sum_{V\in \mv_i} \P(V) \le \P(S)=1.
\]
Therefore, since $N\kappa<1$ by the definition of $N$, we get
\begin{align*}
\P^{\otimes 2} (\Delta E) &\le C(\ep^{1/12} + m^{-1})N + 2\delta_0\\
&\le C(\ep^{1/12} + m^{-1}) \kappa^{-1} + 2\delta_0.
\end{align*}
By the definition \eqref{kappadef} of $\kappa$, this gives the desired result. 
\end{proof}
\begin{lmm}\label{endlmm2}
For any $(x,y)\notin \Delta E$ such that $x\ne y$,
\[
s(x,y) \le ((x,y)_r + 1)\kappa + \delta_0.
\]
\end{lmm}
\begin{proof}
Let $i:= (x,y)_r$, so that $i$ is the largest integer such that $x$ and $y$ both belong to the same member of $\mv_i$. First, suppose that $1\le i\le N-1$ and $s(x,y)\ge t_{i+1}$. Let $V$ be the element of $\mv_i$ that contains $x$ and $y$. Then while applying Lemma \ref{cliquelmm} to $V$, there is an edge between $x$ and $y$ in the original graph, but that edge is deleted in the modification. Thus, $(x,y)\in \Delta E(V)\subset \Delta E$, which is not true by assumption. Therefore $s(x,y)$ must be less than $t_{i+1}$. 

If $i=0$, then also the above deduction holds: If $s(x,y)\ge t_1$ and $x$ and $y$ are both in $S'$, then by the same logic as above we conclude that $(x,y)\in \Delta E$. On the other hand, if $s(x,y)\ge t_1$ and at least one of $x$ and $y$ is outside $S'$, then $(x,y)\in \Delta E(A)\subset \Delta E$.

Combining the above observations, and recalling the bound \eqref{tidef}, we get that if $0\le i\le N-1$, then 
\begin{align*}
s(x,y) &< t_{i+1} \le (i+1)\kappa + \delta_0\\
&= ((x,y)_r +1)\kappa + \delta_0.
\end{align*}
If $i=N$, then note that since $(N+1)\kappa\ge 1$ (by the definition of $N$), 
\[
s(x,y)\le 1\le (N+1)\kappa = ((x,y)_r+1)\kappa. 
\]
Finally, note that since $x\ne y$, we cannot have $i=N+1$.
\end{proof}
\begin{lmm}\label{endlmm3}
For any $(x,y)\notin \Delta E$ such that $x\ne y$, 
\begin{align*}
s(x,y)\ge  (x,y)_r \kappa - \delta_0.
\end{align*}
\end{lmm}
\begin{proof}
As in the proof of Lemma \ref{endlmm2}, let $i := (x,y)_r$, and note that since $x\ne y$, we must have $0\le i\le N$. First, suppose that $2\le i\le N$ and  $s(x,y)< t_i$. We know that $x$ and $y$ are both in some $V\in \mv_i$. Let $U\in \mv_{i-1}$ be the parent of $V$ in $T$. Then while applying Lemma~\ref{cliquelmm} to $U$, $(x,y)$ is not an edge in the original graph, but since $x$ and $y$ both belong to $V$, $(x,y)$ must be an edge in the modified graph. Thus, $(x,y)\in \Delta E(U)\subset \Delta E$, which is false by assumption. Consequently, $s(x,y)\ge t_i$. 

If $i=1$ and $s(x,y)< t_1$, then either $x$ and $y$ are both in $S'$, in which case the same argument shows that $(x,y)\in \Delta E(S') \subset \Delta E$, or at least one of $x$ and $y$ is in $A$, in which case $(x,y)\in \Delta E(A)\subset \Delta E$.  

Combining, and applying \eqref{tidef}, we get that if $1\le i\le N$, then 
\begin{align*}
s(x,y)\ge t_i \ge i \kappa -\delta_0 = (x,y)_r \kappa - \delta_0.
\end{align*}
Lastly, if $i=0$, note that the inequality is automatic since $(x,y)_r=0$. This completes the proof of the lemma.
\end{proof}

\section{Completing the proof of Theorem \ref{mainresult}}
Take any $\eta>0$. We have to prove the existence of a $\gamma>0$, depending only on $\eta$, such that if $P^*< \gamma$ and $\textup{Hyp}(S,\mf,\P, s)< \gamma$, then $\textup{Tree}(S,\mf,\P, s)< \eta$. To do this, first choose $\ep$ so small and $m$ so large that 
\[
C\ep^{1/24}+Cm^{-1/2}\le \frac{\eta}{4},
\]
where $C$ is the universal constant from Lemma \ref{endlmm1}, and also
\[
\kappa = \max\{\ep^{1/24}, m^{-1/2}\}\le \frac{\eta}{4}.
\]
Let  $\delta := \textup{Hyp}(S,\mf,\P, s)$, and let $\delta_0 := \delta^{1/8}$. If $P^*$ and $\delta_0$ are small enough (depending on $\ep$ and $m$), then the method of Section \ref{treesec} yields $\Delta E$ and $T$ satisfying the conclusions of Lemmas~\ref{endlmm1}, \ref{endlmm2} and \ref{endlmm3}. Recall also that $0\le s(x,y)\le 1$ and $0\le (x,y)_r\kappa \le (N+1)\kappa\le 1+\kappa$ for all $x$ and $y$. Consequently, if $X$ and $Y$ are i.i.d.~draws from $\P$, then
\begin{align*}
\E|s(X,Y) - (X,Y)_r\kappa| &\le \kappa + \delta_0 +(1+\kappa)( \P^{\otimes 2}(\Delta E)+ \P(X=Y))\\
&\le \frac{\eta}{4}+\delta_0+\biggl(1+\frac{\eta}{4}\biggr)\biggl(\frac{\eta}{4} + 2\delta_0 + P^*\biggr). 
\end{align*}
This shows that if $P^*$ and $\textup{Hyp}(S,\mf,\P,s)$ are small enough, depending on $\eta$, then $\textup{Tree}(S,\mf,\P, s)<\eta$. 

\section{From Theorem \ref{mainresult} to Theorem \ref{mainresult0}}
In this section we prove Theorem \ref{mainresult0} using Theorem \ref{mainresult}. Initially, let us continue working under the assumption that $S$ is finite and $\mf$ is the power set of $S$. Take any $\ep>0$. Then by Theorem \ref{mainresult}, there is some $\delta>0$ such that if $P^*< \delta$ and $\textup{Hyp}(S,\mf,\P, s)< \delta$, then $\textup{Tree}(S,\mf,\P, s)< \ep$. Suppose that $P^*\ge \delta$. Then we first create a new system where this violation does not happen. Take each $x\in S$ divide it up into $k(x)$ vertices, where $k(x)$ is chosen so large that $\P(x)/k(x)< \delta$. Let $S'$ be the new set of vertices, consisting of $k(x)$ copies of each $x\in S$. Let $f$ be a map from $S'$ into $S$ that takes any copy of $x\in S$ to $x$, so that $|f^{-1}(x)|=k(x)$. Define a probability measure $\P'$ on $S'$ as 
\[
\P'(y) := \frac{\P(f(y))}{k(f(y))}. 
\]
The probability measure $\P'$ can be described in words as follows. Drawing a vertex from $\P'$ is the same as first picking a vertex from $\P$, and then choosing one of its copies in $S'$ uniformly at random. Note that if $Y\sim \P'$, then $f(Y)\sim \P$. 

Define also a similarity function $s'$ on $S'$ as
\[
s'(y,z) := s(f(y), f(z)). 
\]
Then by the observations from the previous paragraph, it follows that 
\[
\textup{Hyp}(S',\mf',\P', s') = \textup{Hyp}(S,\mf,\P,s),
\]
where $\mf'$ is the power set of $S'$. 
On the other hand $\max_{y\in S}\P'(y)<\delta$ by construction. Thus, by Theorem~\ref{mainresult}, 
\[
\textup{Tree}(S',\mf',\P', s')<\ep.
\] 
Consequently, there exists a tree $T'$ that is compatible with $S'$ (in the sense of Definition \ref{treedef}), with root $r$, and a number $\alpha$ such that 
\begin{align}\label{bigtree}
\E|s'(Y,Z) - \alpha (Y,Z)_{r} |<\ep,
\end{align}
where $Y$ and $Z$ are i.i.d.~draws from $\P'$, and $(Y,Z)_r$ is the Gromov product of $Y$ and $Z$ under the graph distance on $T'$, with respect to the base point~$r$. 

Now, for each $x\in S$, let $Y(x)$ be a vertex chosen uniformly at random from $f^{-1}(x)$. Modify the tree $T'$ by deleting all leaves other than the $Y(x)$'s, and also deleting the edges joining these leaves  to their parents. The resulting graph  is still a tree, and its leaves are in one-to-one correspondence with the set $S$. Thus we can relabel its leaves to define a tree $\tilde{T}$ with set of leaves $S$ and root $r$. 

Let $X_1$ and $X_2$ be i.i.d.~draws from $\P$, independent of $\tilde{T}$. Then $Y(X_1)$ and $Y(X_2)$ are i.i.d.~draws from $\P'$, and hence by \eqref{bigtree},
\[
\E|s'(Y(X_1),Y(X_2)) - \alpha (Y(X_1),Y(X_2))_r |<\ep.
\]
But $s'(Y(X_1),Y(X_2)) = s(X_1, X_2)$, and by our definition of $\tilde{T}$,
\begin{align*}
&d_{T'}(Y(X_1),Y(X_2)) = d_{\tilde{T}}(X_1, X_2), \\ 
&d_{T'}(Y(X_1), r) = d_{\tilde{T}} (X_1, r), \ \ d_{T'}(Y(X_2), r) = d_{\tilde{T}} (X_2, r).  
\end{align*}
Therefore $(Y(X_1), Y(X_2))_r = (X_1, X_2)_r$, where the Gromov product on the left is on the tree $T'$, and the Gromov product on the right is on the tree $\tT$. This gives 
\[
\E|s(X_1,X_2) -\alpha  (X_1,X_2)_r |<\ep,
\]
where the expectation is now taken over $X_1$, $X_2$ and $\tilde{T}$. Since $\tilde{T}$ is independent of  $X_1$ and $X_2$, this proves the existence of a tree $T$ with set of leaves $S$ and root $r$, such that 
\[
\E|s(X_1,X_2) - \alpha (X_1,X_2)_r|<\ep.
\]
Thus, we may conclude that $\textup{Tree}(S,\mf, \P,s)<\ep$. This completes the proof of Theorem \ref{mainresult0} under the assumptions that $S$ is finite and $\mf$ is the power set of $S$. 

Let us now consider general $(S,\mf, \P,s)$, where $\mf$ is countably generated. Take any $\ep>0$. The case of finite $S$ gives a $\delta$ corresponding to $\ep/2$. Take this $\delta$, and suppose that 
\begin{align}\label{hypassump}
\textup{Hyp}(S,\mf,\P, s)< \frac{\delta}{2}.
\end{align}
We will show that in the general case, this implies $\textup{Tree}(S, \mf,\P,s)<\ep$.

Let $\{A_1, A_2,\ldots\}$ be a set of generators of $\mf$. For each $n$, let $\cp_n$ be the partition of $S$ generated by $A_1,\ldots, A_n$. Let $\cp_n^{2}$ be the set of all sets of the form $A\times B$ where $A,B\in \cp_n$. Let $\mg_n$ be the set of subsets of $S^2$ that are unions of elements of $\cp_n^{2}$. Define
\[
\mg :=\bigcup_{n=1}^\infty \mg_n.
\]
It is not difficult to show that $\mg$ is an algebra of sets that generates the $\sigma$-algebra $\mf\times \mf$ on $S^2$. Now take any $k\ge 1$. For $0\le j\le k$, let 
\[
B_j := \{(x,y)\in S^2: j/k\le s(x,y)< (j+1)/k\}.
\]
By the measurability of $s$, $B_j\in \mf \times \mf$. Therefore by a basic result of measure theory, given any $\eta>0$ there exists $B_j'\in \mg$ such that $\P^{\otimes 2}(B_j\Delta B_j')\le \eta$. Define
\[
D := \bigcup_{j=0}^k B_j\Delta B_j',
\]
so that $\P^{\otimes 2}(D)\le (k+1)\eta$.

Since $\mg_n$ is an increasing sequence, there is some large enough $n$ such that $B_j'\in \mg_n$ for all $j$. Define a function $\ts :S^2 \to[0,1]$ as $\ts(x,y)=j/k$ where $j$ is a smallest number such that $(x,y)\in B_j'$. If there is no such $j$, let $\ts(x,y)=0$. Since each $B_j'$ is a union of members of $\cp_n^2$, it follows that $\ts$ is constant on each element of $\cp_n^2$. 

Now suppose that $\ts(x,y)=j/k$, but $(x,y)\notin B_j$. Then there are two possibilities: (a) $(x,y)\in B_j'$. Then clearly, $(x,y)\in D$. (b) $(x,y)\notin B_j'$. In this case, $j$ must be zero and $(x,y)$ must not belong to any $B_i'$. But $(x,y)\in B_i$ for some $i$. Thus again, $(x,y)\in D$.

On the other hand, suppose that $(x,y)\in B_j$ but $\ts(x,y)\ne j/k$. Again, this implies that either $(x,y)$ is not in any $B_i'$, or $(x,y)\in B_i'$ for some $i\ne j$. In the first case, we clearly have $(x,y)\in D$. In the second, $(x,y)\notin B_i$ and hence $(x,y)\in D$. 

Combining the observations of the last two paragraphs, we see that if $|\ts(x,y)-s(x,y)|>1/k$, then $(x,y)\in D$. Thus, if $X$ and $Y$ are i.i.d.~draws from $\P$, then
\begin{align}\label{tsapprox}
\E|\ts(X,Y) - s(X,Y)|\le \frac{1}{k} + \P^{\otimes 2}(D)\le \frac{1}{k}+(k+1)\eta.
\end{align}
Now recall  the assumption \eqref{hypassump} and the fact that $\delta$ is a function of $\ep$. Therefore, the above display shows that by choosing $k$ large enough (depending on $\ep$), and then choosing $\eta$ small enough (depending on $k$ and $\ep$), we can  ensure that 
\[
\textup{Hyp}(S, \mf, \P, \ts)< \delta. 
\]
Now let $\tx$ be the element of $\cp_n$ that contains $X$ and let $\ty$ be the element of $\cp_n$ that contains $Y$. Since $\cp_n$ is a finite set, we can endow it with its power set $\sigma$-algebra $2^{\cp_n}$ (which identifies with $\mg_n$), and may consider $\tx$ and $\ty$ to be $\cp_n$-valued random variables. Then $\tx$ and $\ty$ are i.i.d.~random variables with law $\tp$, where $\tp$ identifies with the restriction of $\P$ to $\mg_n$. Since $\ts$ is constant on elements of $\cp_n^2$, we can naturally view $\ts$ as a function on $\cp_n\times \cp_n$.  Lastly, observe that $\ts(\tx,\ty) = \ts(X,Y)$. Combining all of these observations, we get
\[
\textup{Hyp}(\cp_n, 2^{\cp_n}, \tp, \ts) = \textup{Hyp}(S, \mf, \P, \ts) < \delta. 
\]
Since $\cp_n$ has finite cardinality, this implies that 
\[
\textup{Tree}(\cp_n, 2^{\cp_n}, \tp, \ts) < \frac{\ep}{2}. 
\]
In particular, there is a tree $\tT$ with root $r$ that is compatible with $(\cp_n, 2^{\cp_n})$, and a number $\alpha \ge 0$, such that
\begin{align}\label{tstxty}
\E|\ts(\tx,\ty)- \alpha(\tx,\ty)_r| < \frac{\ep}{2},
\end{align}
where $(\tx,\ty)_r$ is the Gromov product of $\tx$ and $\ty$ under the graph distance on $\tT$, with respect to the base point $r$. Let us now extend the tree $\tT$ by appending $S$ to the set of nodes, and adding an edge between each $x\in S$ and the element of $\cp_n$ that contains $x$. Call the new tree $T$. Then $S$ is the set of leaves of $T$. The set $T\setminus S$ is just $\tT$, which is finite. Lastly, for any $v\in T\setminus S$, the set of leaves that are descendants of $v$ is a union of elements of $\cp_n$, and therefore measurable. Thus, $T$ is compatible with $(S,\mf)$. 

Next, note that $(\tx, \ty)_r = (X, Y)_r$, because if $d_T$ is the graph distance on $T$, then $d_T(X, r) = d_{\tT}(\tx, r)+1$, $d_T(Y,r)=d_{\tT}(\ty,r)+1$, and $d_T(X,Y) = d_{\tT}(\tx, \ty)+2$. Also, we know that $\ts(\tx,\ty)=\ts(X,Y)$. Therefore by \eqref{tstxty},
\begin{align*}
\E|\ts(X,Y)- \alpha(X,Y)_r| < \frac{\ep}{2}.
\end{align*}
Invoking \eqref{tsapprox}, this shows that if $k$ is chosen large enough (depending on $\ep$), and then $\eta$ is chosen small enough (depending on $k$ and $\ep$), we can ensure that 
\[
\E|s(X,Y)- \alpha(X,Y)_r| < \ep.
\]
Consequently, $\textup{Tree}(S, \mf, \P, s)<\ep$, completing the proof of Theorem \ref{mainresult0}.

\section{Proof of Theorem \ref{spinthm}}\label{spinproof}
Take any strictly increasing continuous function $\rho:\R\to [0,\infty)$, and define the similarity function
\[
s_n(\sigma^1, \sigma^2) := \rho(f(R_{1,2})).
\]
If three configurations $\sigma^1$, $\sigma^2$ and $\sigma^3$ satisfy 
\[
f(R_{1,2}) \ge \min\{f(R_{1,3}), f(R_{2,3})\}-\ep
\]
for some $\ep\ge 0$, then by the monotonicity and uniform continuity of $\rho$ on the range of $f$, 
\begin{align*}
\rho(f(R_{1,2})) &\ge \rho(\min\{f(R_{1,3}), f(R_{2,3})\}-\ep) \\
&\ge  \rho(\min\{f(R_{1,3}), f(R_{2,3})\}) - \delta(\ep) \\
&= \min\{\rho(f(R_{1,2})), \rho(f(R_{1,3})) \} - \delta(\ep),
\end{align*}
where $\delta(\ep)\to 0$ as $\ep\to 0$. From this and the boundedness of $\rho$ on the range of $f$, we see that if \eqref{panchenkogen} holds, then 
\[
\lim_{n\to\infty} \E\smallavg{(\min\{\rho(f(R_{1,3})), \rho(f(R_{2,3}))\}-\rho(f(R_{1,2})))_+} = 0.
\]
Consequently, $\textup{Hyp}(\Sigma_n, \mf_n, \mu_n, s_n) \to 0$ in probability as $n\to\infty$, where $\mf_n$ is the power set of $\Sigma_n$ if $\Sigma_n = \{-1,1\}^n$ and the Borel $\sigma$-algebra of $\Sigma_n$ if $\Sigma_n = \sqrt{n}\mathbb{S}^{n-1}$. 
Thus, Theorem \ref{mainresult0} implies that 
\begin{align*}
\textup{Tree}(\Sigma_n,\mf_n,\mu_n, s_n)\to 0 \text{ in probability as $n\to\infty$.}
\end{align*}
Therefore, there are sequences $\ep_n$ and $\delta_n$ tending to zero as $n\to\infty$, such that the following holds. With probability at least $1-\ep_n$, there exists a tree $T_n$ with root $r_n$, that is compatible with $(\Sigma_n,\mf_n)$ in the sense of Definition~\ref{treedef}, and a number $a_n\ge 0$, satisfying
\[
\avg{|\rho(f(R_{1,2})) - a_n(\sigma^1, \sigma^2)_{r_n}|}\le  \delta_n,
\]
where  $(\sigma^1, \sigma^2)_{r_n}$ is the Gromov product under graph distance on the tree $T_n$, with respect to the  base point $r_n$.

By the remark immediately below Definition \ref{treedef}, the nodes of $T_n$ give a hierarchical clustering of $\Sigma_n$ into measurable clusters. For each node $\alpha$, let $q_\alpha := \rho^{-1}(a_n d_\alpha)$, where $d_\alpha$ is the length of path from $r_n$ to $\alpha$. If $\alpha$ is the smallest cluster containing $\sigma^1$ and $\sigma^2$, then $(\sigma^1,\sigma^2)_{r_n} =d_\alpha$. Therefore if $\rho(f(R_{1,2})) \approx a_n (\sigma^1, \sigma^2)_{r_n}$, then $f(R_{1,2})\approx q_\alpha$. This completes the proof.

\section*{Acknowledgements}
We thank Sky Cao, Wei-Kuo Chen, Persi Diaconis, Jacob Fox, Susan Holmes and  Dmitry Panchenko for helpful comments and references.

\end{document}